\documentclass[11pt]{amsart}

\usepackage{amsmath,amsthm, amscd, amssymb, amsfonts}
\usepackage[all]{xy}

\newtheorem{theorem}{Theorem}[section]
\newtheorem{lem}[theorem]{Lemma}
\newtheorem{cor}[theorem]{Corollary}

\newtheorem{pro}[theorem]{Proposition}
\newtheorem*{claim}{Claim}

\theoremstyle{definition}
\newtheorem{fed}[theorem]{Definition}

\theoremstyle{remark}
\newtheorem{rem}[theorem]{Remark}

\newcommand{\ydh}{{}^{H}_{H}\mathcal{YD}}
\newcommand{\ydg}{{}^{G}_{G}\mathcal{YD}}

\newcommand\id{\operatorname{id}}

\newcommand\Hom{\operatorname{Hom}}
\newcommand\End{\operatorname{End}}
\newcommand\ind{\operatorname{Ind}}
\newcommand\ext{\operatorname{Ext}}
\newcommand\mo{\operatorname{mod}}
\newcommand\sg{\operatorname{sgn}}
\newcommand\st{\operatorname{st}}

\newcommand\gr{\operatorname{gr}}

\newcommand\ad{\operatorname{ad}}

\newcommand\diag{\operatorname{diag}}

\newcommand\img{\operatorname{im}}

\newcommand\ab{\operatorname{ab}}

\def\k{\Bbbk}
\def\yd{\mathfrak{YD}}

\def\ot{\otimes}

\def\s{\mathbb{S}}

\def\N{\mathbb{N}}
\def\B{\mathfrak{B}}

\def\eps{\epsilon}
\def\Z{\mathbb{Z}}

\def\mA{\mathcal{A}}

\def\mP{\mathcal{P}}
\def\mO{\mathcal{O}}
\def\mR{\mathcal{R}}

\def\mQ{\mathcal{Q}}
\def\mH{\mathcal{H}}
\def\hq{\mH(\mQ)}
\def\im{\mathrm{i}}

\begin{document}


\title[Representations of pointed Hopf algebras over
$\s_3$] {Representations of finite dimensional pointed Hopf algebras
over $\s_3$}

\author[Garc\'ia Iglesias, Agust\'in]
{Garc\'ia Iglesias, Agust\'in}

\address{FaMAF-CIEM (CONICET), Universidad Nacional de C\'ordoba,
Medina A\-llen\-de s/n, Ciudad Universitaria, 5000 C\' ordoba,
Rep\' ublica Argentina.} \email{aigarcia@mate.uncor.edu}

\thanks{\noindent 2000 \emph{Mathematics Subject Classification.}
16W30. \newline The work was partially supported by CONICET,
FONCyT-ANPCyT, Secyt (UNC)}

\date{\today}

\begin{abstract}
The classification of finite-dimensional pointed Hopf algebras with group $\s_3$ was finished in
\cite{AHS}: there are exactly two of them, the bosonization of a Nichols algebra of dimension 12 and a non-trivial lifting.
Here we determine all simple modules over any of these
Hopf algebras. We also find the Gabriel
quivers, the projective covers of the simple modules, and prove
that they are not of finite representation type.
To this end, we first investigate the modules over some complex pointed
Hopf algebras defined in the papers \cite{AG1, GG},
whose restriction to the group of group-likes is a direct sum of
1-dimensional modules.
\end{abstract}

\maketitle

\section{Introduction}
In \cite{AG1}, a pointed Hopf
algebra $H_n$ was defined for each $n\geq 3$. It was shown there that $H_3$ and $H_4$ are non-trivial pointed Hopf algebras over $\s_3$ and $\s_4$, respectively. We showed in \cite{GG} that this holds for every $n$, by different methods. We
started by defining generic families of pointed Hopf algebras
associated to certain data, which includes a finite non-abelian
group $G$. Under certain conditions, these algebras are liftings
of (possibly infinite dimensional) quadratic Nichols algebras over
$G$. In particular, this was proven to hold for $G=\s_n$.
Moreover, the classification of finite
dimensional pointed Hopf algebras over $\s_4$ was finished. We review some
of these facts in Section 2. We investigate, in Section 3, modules
over these algebras whose $G$-isotypic components are 1-dimensional
and classify indecomposable modules of this kind. We find
conditions on a given $G$-character under which it can be extended
to a representation of the algebra. We apply these results to the
representation theory of two families of pointed Hopf algebras
over $\s_n$. In Section 4 we comment on some known facts about simple modules
over bosonizations. We also prove general facts about projective
modules over the algebras defined in \cite{AG1,GG}, and recall a few facts about
representation type of finite dimensional algebras. In Section 5 we use some of
the previous results to classify simple modules over pointed Hopf algebras over
$\s_3$. In addition, we find their
projective covers and compute their fusion rules, which lead to show that
the non-trivial lifting is not quasitriangular. We also
write down the Gabriel quivers and show that these algebras are
not of finite representation type. 

\section{Preliminaries}
We work over an algebraically closed field $\k$ of characteristic
zero. We fix $\im=\sqrt{-1}$. For $n\in\mathbb{N}$, let $[\frac n2]$ denote the
biggest integer lesser or equal than $\frac n2$. If $V$ is a vector space and
$\{x_i\}_{i\in I}$ is a family of elements in $V$, we denote by
$\k\{x_i\}_{i\in I}$ the vector subspace generated by it.  Let
$G$ be a finite group, $\widehat{G}$ the set of its irreducible
representations. Let $G_{\ab}=G/[G,G]$,
$\widehat{G_{\ab}}=\Hom(G,\k^*)\subseteq \widehat{G}$. We denote by
$\eps\in\widehat{G_{\ab}}$ the trivial representation. If
$\chi\in\widehat{G}$, and $W$ is a $G$-module, we denote by
$W[\chi]$ the isotypic component of type $\chi$, and by $W_\chi$ the
corresponding simple $G$-module.

A \emph{rack} is a pair $(X,\rhd)$, where $X$ is a non-empty set
and $\rhd:X\times X\to X$ is a function, such that $\phi_i=i\rhd
(\cdot):X\to X$ is a bijection for all $i\in X$ and $ i\rhd(j\rhd
k)=(i\rhd j)\rhd (i\rhd k), \, \forall i,j,k\in X$. A rack $(X,\rhd)$ is said to
be \emph{indecomposable} if it cannot
be decomposed as the disjoint union of two sub-racks. We shall always
work with racks that are in fact quandles, that is that $i\rhd i=i$
$\forall\,i\in X$.
In practice, we are interested in the case in which the rack $X$ is a conjugacy
class in a group;
hence this  assumption always holds. We will denote by 
$\mO_2^n$ the conjugacy class of transpositions in $\s_n$. 

A 2-cocycle
$q:X\times X\to \k^*$, $(i,j)\mapsto q_{ij}$ is a function such
that $q_{i,j\rhd k}q_{j,k}=q_{i\rhd j,i\rhd k}q_{i,k}, \
\forall\,i,j,k\in X$. See \cite{AG1} for a detailed exposition on this matter.

Let $H$ be a Hopf
algebra over $\k$, with antipode $\mathcal S$. Let $\ydh$ be the category of (left-left)
Yetter-Drinfeld modules over $H$. That is, $M$ is an object of $\ydh$ if and
only if there exists an action $\cdot$ such that $(M,\cdot)$ is a (left)
$H$-module and a coaction $\delta$ such that $(M,\delta)$ is a
(left) $H$-comodule, subject to the following compatibility
condition:
\begin{equation*}
\delta(h\cdot m)=h_1m_{-1}\mathcal S(h_3)\ot h_2\cdot m_{0}, \ \forall\, m\in
M, h\in H,
\end{equation*}
where $\delta(m)=m_{-1}\ot m_0$. If $G$ is a finite group and
$H=\k G$, we write $\ydg$ instead of $\ydh$.

Recall from \cite[Def. 3.2]{AG2} that a \emph{principal
YD-realization} of $(X,q)$ over a finite group $G$ is a collection
$(\cdot, g, (\chi_i)_{i\in X})$ where 
\begin{itemize}
\item  $\cdot$ is an action of $G$
on $X$;
\item  $g:X\to G$ is a function such that $g_{h\cdot i} = hg_
{i}h^{-1}$ and $g_{i}\cdot j=i\rhd j$;
\item the family
$(\chi_i)_{i\in X}$, with $\chi_i:G\to\k^*$, is a 1-cocycle,
\emph{i.~e.} $\chi_i(ht)=\chi_i(t)\chi_{t\cdot i}(h)$, for all
$i\in X$, $h,t\in G$, satisfying $\chi_i(g_{j})=q_{ji}$.
\end{itemize}
In words, a principal YD-realization over $G$ is a way to realize the braided
vector space $(\k X,c^q)$ as a YD-module over $G$. See \cite{AG2} for details.

\subsection{Quadratic lifting data}

\

Let $X$ be a rack, $q$ a 2-cocycle. Let $\mR$ be the set of
equivalence classes in $X\times X$ for the relation generated by
$(i,j)\sim(i\rhd j,i)$. Let $C\in\mR$, $(i,j)\in C$. Take $i_1=j$,
$i_2=i$, and recursively, $i_{h+2}=i_{h+1}\rhd i_h$. Set $n(C)=\#C$
and $$\mR'=\Big\{C\in \mR\,|\,\prod_{h=1}^{n(C)}
q_{i_{h+1},i_h}=(-1)^{n(C)}\Big\}.$$ Let $F$ be the free associative
algebra in the variables $\{T_l\}_{l\in X}$. If $C\in\mR'$, consider
the quadratic polynomial \begin{equation}\label{eqn:phiC}\phi_C =
\sum_{h=1}^{n(C)}\eta_h(C) \, T_{i_{h+1}}T_{i_h}\in F,\end{equation}
where $\eta_1(C)=1$ and
$\eta_h(C)=(-1)^{h+1}q_{{i_2i_1}}q_{{i_3i_2}}\ldots
  q_{{i_hi_{h-1}}}$, $h\ge 2$.

\smallbreak
A \emph{quadratic lifting datum} $\mQ=(X,q,G, (\cdot, g,
(\chi_l)_{l\in X}),(\lambda_C)_{C\in \mR'})$, or ql-datum,
\cite[Def. 3.5]{GG}, is a collection consisting of 
\begin{itemize}
 \item a rack $X$; \item a
2-cocycle $q$; \item a finite group $G$; \item a principal YD-realization
$(\cdot, g, (\chi_l)_{l\in X})$ of $(X,q)$ over $G$ such that
$g_i\neq g_jg_k, \, \forall\,i,j,k\in X$; \item a collection
$(\lambda_C)_{C\in \mR'}\in\k$ such that, if
$C=\{(i_2,i_1),\ldots,(i_n,i_{n-1})\}$, and $k\in X$,
\begin{align}
   \lambda_C&=0, &&  \mbox{if } \ g_{i_2}g_{i_1}=1,\label{eqn:lambda1}\\
  \label{eqn:lambda2} \lambda_C&=q_{k
i_{2}}q_{k i_{1}}\lambda_{k\rhd C},  &&
\end{align}
where $k\rhd C=\{(k\rhd i_2,k\rhd i_1),\ldots,(k\rhd i_n,k\rhd i_{n-1})\}$.
\end{itemize}
In \cite{GG}, we attached a pointed Hopf algebra $\mH(\mQ)$
to each ql-datum $\mQ$. It is generated by $\{a_l, H_t : l\in X, \,
t\in G\}$ with relations:
\begin{align}
\label{eqn:hq1} H_e &=1, \quad H_tH_s=H_{ts}, &  t,s&\in G;\\
\label{eqn:hq2} H_ta_l &= \chi_l(t)a_{t\cdot l}H_t,  &  t&\in G, \, l \in X; \\
\label{eqn:hq3} \phi_C(\{a_l\}_{l\in X}) &= \lambda_C(1-H_{g_ig_j}),  &
C&\in \mR', \, (i,j)\in C.
\end{align}
Here $\phi_C$ is as in \eqref{eqn:phiC}
above. We denote by $a_C$ the left-hand
side of \eqref{eqn:hq3}. $\mH(\mQ)$ is a pointed Hopf algebra, setting
$\Delta(H_t)=H_t\ot H_t$,  $\Delta(a_i)=g_i\ot a_i+a_i\ot 1$, $t\in
G$, $i\in X$. See \cite{GG} for further details on this construction
and for unexplained terminology.

Notice that by definition of the Hopf algebras $\hq$, the group of
grouplikes $G(\hq)$ is a quotient of the group $G$. Thus, any
$\hq$-module $M$ is a $G$-module, using the corresponding
projection. We denote this module by $M_{|G}$. For simplicity, we
denote $M[\rho]=M_{|G}[\rho]$, $\rho\in\widehat G$.

\section{Modules that are sums of 1-dimensional representations}

In this Section, we study $\hq$-modules whose underlying $G$-module
is a direct sum of representations in $\widehat{G_{\ab}}$. 

We begin by fixing the following notation. Given a pair $(X,q)$,
let
\begin{equation}\label{eqn:zh}
 \zeta_h(C)=\begin{cases}

(-1)^{\frac{h}2-1}\left(\prod\limits_{l=1}^{\frac{h}2-1}q_{i_{h-2l+1},i_{h-2l}}
\right) & \text{if } 2|h,\\

(-1)^{\frac{h-1}2}\left(\prod\limits_{l=1}^{\frac{h-1}2}q_{i_{h-2l+1},i_{h-2l}}
\right) & \text{if } 2|h+1.
         \end{cases}
\end{equation}
Note that $\zeta_{1}(C) = \zeta_2(C) = 1$, $\zeta_{h+1}(C)\zeta_h(C)
=\eta_h(C)$, see \eqref{eqn:phiC}.

\subsection{Modules whose underlying $G$-module is isotypical}

\

We first study extensions of multiplicative characters from $G$ to $\hq$.

\begin{pro}\label{pro:dim1}
Let $\rho\in\widehat {G_{ab}}$. There exists
$\bar\rho\in\hom_{alg}(\hq,\k)$ such that $\bar\rho_{|G}=\rho$ if and only if
\begin{equation}\label{eqn:tres}
  0=\lambda_C(1-\rho(g_ig_j))  \ \text{if }(i,j)\in C\text{ and }
2|n(C),
\end{equation}
and there exists a family $\{\gamma_i\}_{i\in X}$ of scalars such that
\begin{align}
\gamma_j&=\chi_j(t)\gamma_{t\cdot j}&&\forall\,t\in G,j\in
X,\label{eqn:uno}\\
\gamma_i\gamma_j&=\lambda_C(1-\rho(g_ig_j))&&\text{if
}(i,j)\in
C\text{ and } 2|n(C)+1\label{eqn:dos}.
\end{align}
If \eqref{eqn:tres} holds, then the set of all extensions $\bar\rho$
of $\rho$ is in bijective correspondence with the set of families
$\{\gamma_i\}_{i\in X}$
that satisfy \eqref{eqn:uno} and \eqref{eqn:dos}. In particular, if
\begin{equation}\label{eqn:gigj}   \lambda_C\neq
0\Rightarrow\rho(g_ig_j)=1, \quad C\in\mR',\,(i,j)\in C.
\end{equation}
then $\gamma_i=0,\ \forall\,i\in X$ defines an $\hq$-module.
Moreover, this is the only possible extension if, in addition,
\begin{equation}\label{eqn:chiandrhd} \chi_i(g_i)\neq
1,\quad \forall\,i\in X.\end{equation}
\end{pro}
\begin{rem}\label{rem:ii}
(a) Mainly, we will deal with Nichols algebras for which the following is
satisfied:
\begin{equation}\label{eqn:chiandrhd1} \chi_i(g_i)=
-1,\quad \forall\,i\in
X.\end{equation} \noindent  In this case, obviously \eqref{eqn:chiandrhd} holds
and the class $C_i=\{(i,i)\}$ belongs to $\mR'$.

\medbreak (b) If $X$ is indecomposable, using \eqref{eqn:uno} and the fact that
$\forall\, i\in X$ $\exists\,t\in G$ such that $i=t\cdot j$, we may
replace \eqref{eqn:dos} by
 \begin{align}\tag{\ref{eqn:dos}'}
\gamma_j^2 = \lambda_C(1-\rho(g_j)^2)\chi_j(t)&&\text{if
 } (i,j)\in C\text{ and } 2|n(C)+1.
 \end{align}
\end{rem}

\bigbreak\begin{proof}
Assume that such $\bar\rho$ exists and let $\gamma_i=\bar\rho(a_i)$.
Then \eqref{eqn:uno} follows from \eqref{eqn:hq2}. In particular,
for $p,q\in X$, we have $\bar\rho(a_{p\rhd
q})=\chi_q(g_p)^{-1}\bar\rho(a_q)$. Then, for $C\in\mR'$,
$(i_2,i_1)=(i,j)\in C$, it follows that
\begin{equation}\label{eqn:gamma}
  \gamma_{i_h}=\bar\rho(a_{i_h})=\begin{cases}
    (-1)^{\frac{h-1}2}\zeta_h(C)^{-1}\bar\rho(a_j)&\text{if }2|h+1\\
        (-1)^{\frac{h}2-1}\zeta_h(C)^{-1}\bar\rho(a_i)&\text{if }2|h,\\
  \end{cases}
\end{equation}
cf. \eqref{eqn:zh}. Consequently, \begin{equation}\label{eqn:etah}
\bar\rho(a_{i_{h+1}}a_{i_{h}})=(-1)^{h+1}\eta_h(C)^{-1}
\bar\rho(a_i)\bar\rho(a_j)
\end{equation} and thus \eqref{eqn:dos} and
\eqref{eqn:tres} follow from \eqref{eqn:hq3}. Conversely, if \eqref{eqn:tres}
holds and $\{\gamma_i\}_{i\in X}$ is a family that satisfies
\eqref{eqn:uno} and \eqref{eqn:dos}, then we define $\bar\rho:\hq\to\k$ as the
unique algebra morphism such that
$\bar\rho(H_t)=\rho(t)$ and $\bar\rho(a_i)=\gamma_i$. If \eqref{eqn:chiandrhd}
holds, it follows from \eqref{eqn:uno} for
$t=g_i$ that $\bar\rho(a_i)=0\,\forall\,i\in X$ is a necessary
condition.
\end{proof}
\begin{fed}
Let $\bar\rho$ be an extension of $\rho\in\widehat {G_{ab}}$
and $\gamma_i=\bar\rho(a_i)$, $\gamma=(\gamma_i)_{i\in
X}\in\k^{X}$. Then we denote the corresponding $\hq$-module by
$S_\rho^\gamma$. If $\gamma=0$, we set $S_\rho^\gamma=S_\rho$.
\end{fed}

We now determine all $\hq$-modules whose underlying $G$-module is isotypical of type $\rho\in\widehat {G_{ab}}$,
provided that $X$ is indecomposable and \eqref{eqn:chiandrhd} holds.

\begin{pro}\label{pro:dims1}
Assume $X$ is indecomposable. Let  $M$ be an $\hq$-module such that $M=M[\rho]$
for a unique $\rho\in\widehat {G_{ab}}$, $\dim M=n$.
Then $M$ is simple if and only if $n=1$. If, in addition,
\eqref{eqn:chiandrhd} holds, $M\cong S_\rho^{\oplus n}$.
\end{pro}

\begin{proof}
Let $\bar\rho:\hq\to\End M$ be the
corresponding representation and $\Gamma_j\in\k^{n\times n}$ be the matrix
associated to $\bar\rho(a_j)$ in some (fixed) basis. As in the proof of Prop.
\ref{pro:dim1}, $\{\Gamma_i\}_{i\in X}$ satisfies \eqref{eqn:uno}. Thus, if we
fix $j\in X$, then for each $i\in X$ there exists $t\in G$ such that
$\Gamma_i=\chi_j(t)^{-1}\Gamma_j$.  Thus, there exists a basis
$\{z_1,\ldots,z_n\}$ in which all of these matrices are upper
triangular and so $\k\{z_1\}$ generates a submodule $M'\subseteq M$. If
\eqref{eqn:chiandrhd} holds, then it
follows that $\Gamma_i=0$, $\forall\,i\in X$ and thus
$M\cong\bigoplus_{j=1}^nS_\rho$.
\end{proof}

\bigbreak
\subsection{Modules whose underlying $G$-module is a sum of two isotypical
components}\label{subsect:2isotipicas}

\

Let $\rho, \mu\in \widehat{G_{ab}}$ fulfilling \eqref{eqn:tres},
$\gamma,\delta\in\k^X$ satisfying \eqref{eqn:uno} and
\eqref{eqn:dos} for $\rho$ and $\mu$, respectively. We begin this Subsection by
describing
indecomposable modules that are extensions of $S_\rho^\gamma$ by
$S_\mu^\delta$.
For simplicity of the statement of \eqref{eqn:dosb} in the following Lemma, we
introduce the following notation. Let $C\in\mR'$, $j\in C$ and let
$$
\alpha_j(C)=\sum_{r=0}^{[\frac{n(C)}2]-1}\chi_j(g_j)^r, \quad
\beta_j(C)=\sum_{r=0}^{[\frac{n(C)+1}2]-1}\chi_j(g_j)^r.
$$
Note that if $2|n(C)$, then $\alpha_j=\beta_j$; otherwise,
$\beta_j=\alpha_j+\chi_j(g_j)^{[\frac{n(C)+1}2]-1}$.
\begin{lem}\label{lem:wrhomu}
Let $V$ be the space of solutions $\{f_i\}_{i\in X}\in\k^{X}$ of the
following system
\begin{align}
f_i\mu(t)&=\chi_i(t)f_{t\cdot i}\rho(t), \qquad\qquad i\in X,\, t\in
  G \text{ and}\label{eqn:unob}\\
(\alpha_j(C)\delta_j-\beta_j(C)\gamma_j)f_i&=-\chi_i(g_i)(\alpha_i(C)\delta_i-
\beta_i(C)\gamma_i)f_j, \label{eqn:dosb}
\end{align}
$C\in\mR',\,(i,j)\in C.$ Then
$\ext^1_{\hq}(S^\gamma_\rho,S^\delta_\mu)\cong V$ and the set of
isomorphism classes of indecomposable $\hq$-modules such that
\begin{equation}\label{eqn:exacta}
  0\longrightarrow S_\mu^\delta\longrightarrow
M\longrightarrow
  S_\rho^\gamma\longrightarrow0 \text{ is exact}
\end{equation}
is in bijective correspondence with $\mathbb{P}_k(V)$.
\end{lem}
\begin{proof}
Let $M=\k\{z,w\}$ be as in \eqref{eqn:exacta}, with $z\in M[\rho]$, $w\in
M[\mu]$. Then there exists $\{f_i\}_{i\in X}$ such that
\begin{equation}\label{eqn:aiz}a_i z=\gamma_iz+f_iw.
\end{equation} Then \eqref{eqn:unob} follows from
\eqref{eqn:hq2} and this implies
\begin{equation*}
 f_{i_h}=\begin{cases}
          (-\chi_j(g_j))^{\frac{h}2-1}\zeta_h(C)^{-1}f_i & \text{if } 2|h,\\
          (-\chi_i(g_i))^{\frac{h-1}2}\zeta_h(C)^{-1}f_j & \text{if }
          2|h+1,
         \end{cases}
\end{equation*}
since, for $\tau=\rho$ or $\tau=\mu$,
\begin{align*}
 \tau(g_{i_{2l+1}})&=\tau(g_{i_{2l}}g_{i_{2l-1}}g_{i_{2l}}^{-1})=\tau(g_{i_{2l-1
}})=\dots=\tau(g_{i_{1}})=\tau(g_j),\\
 \tau(g_{i_{2l+2}})&=\tau(g_{i_{2l+1}}g_{i_{2l}}g_{i_{2l+1}}^{-1})=\tau(g_{i_{2l
}})=\dots=\tau(g_{i_{2}})=\tau(g_i),
\end{align*}
and $\frac{\mu(g_k)}{\rho(g_k)}=\chi_k(g_k)$. Therefore, if
$(i,j)\in C$ and $n=n(C)$, \eqref{eqn:hq3} holds if and only if
$$\sum_{h=1}^{n}\eta_h(C) \,
  \big(f_{i_h}\delta_{i_{h+1}}+f_{i_{h+1}}\gamma_{i_h}\big)=0,
\forall\,C\in\mR',$$ that is, using \eqref{eqn:gamma},
\eqref{eqn:hq3} holds if and only if \eqref{eqn:dosb} follows.

Conversely, if $\{f_i\}_{i\in X}$ fulfills \eqref{eqn:unob} and
\eqref{eqn:dosb}, then \eqref{eqn:aiz} together with $a_i
w=\delta_iw$ define an $\hq$-module which is an extension of
$S_\rho^\gamma$ by $S_\mu^\delta$. 

$M$ is indecomposable if and only
if $f_i\neq 0$ for some $i\in X$. Assume $M$ is indecomposable and let $M'=\k\{z',w'\}$ be another indecomposable $\hq$-module fitting in \eqref{eqn:exacta}, with $z'\in M'[\rho]$, $w'\in M'[\mu]$. Let $\{g_i\}_{i\in X}\in V$ be the corresponding solution of \eqref{eqn:unob} and \eqref{eqn:dosb}. Assume $\phi:M\to M'$ is an isomorphism of $\hq$-modules. In particular, $\phi$ is a $G$-isomorphism and thus there exist $\sigma,\tau\in\k^*$ such that $\phi(w)=\sigma w'$, $\phi(z)= \tau z'$. But then it is readily seen that $\sigma,\tau$ must satisfy $g_i=\sigma\tau^{-1} f_i$, $i\in X$. That is, $[f_i]_{i\in X}=[g_i]_{i\in X}$ in $\mathbb{P}_\k(V)$. The converse is clear.
\end{proof}

\begin{rem}\label{rem:wrhomu}
If $X$ is indecomposable, then, up to isomorphism, there is at most
one indecomposable $\hq$-module $M$ as in the Lemma. In fact, if
there is one, let $\{f_i\}_{i\in X}\in\k^{X}$ be the corresponding
solution of \eqref{eqn:unob} and \eqref{eqn:dosb}. Then, if we fix
$j\in X$ and let $t_i\in G$ be such that $i=t_i\cdot j$, $i\in X$,
then
\begin{equation}\label{eqn:fifj}
(f_i)_{i\in X}=
f_j\left(\chi_{j}(t_i)\frac{\mu(t_i)}{\rho(t_i)}\right)_{i\in
X}\in\k^{X},
\end{equation}
and thus $M$ is uniquely determined. In this case, the existence of a solution is
equivalent to \eqref{eqn:unob} and
\begin{align}\tag{\ref{eqn:dosb}'}\label{eqn:dosbprima}
(\alpha_j\delta_j-\beta_j\gamma_j)\left(\frac{\mu(t_i)}{\rho(t_i)}
+\chi_j(g_j)\right) f_j=0;
\end{align}
if $(i,j)\in C$, $C\in\mR'$, $i=t_i \cdot j$.
\end{rem}

\begin{fed}
Assume $X$ is indecomposable and $\ext^1_{\hq}(S^\gamma_\rho,S^\delta_\mu)\neq
0$. We
denote the corresponding unique indecomposable $\hq$-module
by $M_{\rho,\mu}^{\gamma,\delta}$. If $\gamma=\delta=0$, then
\eqref{eqn:dosbprima} is a tautology. We set
$M_{\rho,\mu}:=M_{\rho,\mu}^{0,0}$.
\end{fed}

\emph{Assume that $X$ is indecomposable and that $G=\langle\{g_i\}_{i\in X}\rangle$. }Let $j$ be a fixed element in $X$.
Define $\ell:G\to \Z$, resp. $\psi:G\to \k^*$, as
$$
\ell(t)=  \min\{n\,:\,t=g_{i_1}\dots g_{i_n},\ i_1,\dots,i_n\in X\},
$$
resp. $\psi(t) = \chi_j(g_j)^{\ell(t)}$, $t\in G$.
Notice that $\tau(g_i)=\tau(g_j)$, $\forall\,i\in X$, hence
$\tau(t) = \tau(g_j)^{\ell(t)}$, for any $\tau\in\widehat{G_{ab}}$, $t\in G$.

\begin{lem}\label{lem:wrhomu2}
Keep the above hypotheses. If $\ext^1_{\hq}(S^\gamma_\rho,S^\delta_\mu)\neq 0$, then
\begin{align}\label{eqn:mu-determina-rho}
\mu(s)&=\psi(s) \rho(s), \qquad \forall s\in G.
\end{align}
Therefore $\rho$ determines $\mu$ (and vice versa), and $\psi$ is a group
homomorphism.

Conversely, if \eqref{eqn:mu-determina-rho} holds, we may
replace \eqref{eqn:unob}
and \eqref{eqn:dosb} by
\begin{align}
f_i\chi_j(g_j)^{\ell(t)}&=\chi_i(t)f_{t\cdot i}, && i\in X,\, t\in
  G \text{ and}\tag{\ref{eqn:unob}'}\\
0&=
f_j(\alpha_j\delta_j-\beta_j\gamma_j)\left(\chi_j(g_j)^{\ell(t_i)-1}+1\right),
&&
\tag{\ref{eqn:dosb}''}
\end{align}
if $(i,j)\in C$, $C\in\mR'$, $i=t_i \cdot j$.
\end{lem}

\begin{proof}
Setting $i=j$ and $t=g_j$ in \eqref{eqn:unob}, and taking the $\ell(s)$-th power, we get \eqref{eqn:mu-determina-rho}. The rest is straightforward.
\end{proof}

\medbreak

We will show next that there are no simple modules $M$ of dimension 2 such that $M_{|G}$ is sum of
two (necessarily different) components of dimension 1, provided that the
following holds:
\begin{equation}\label{eqn:c1}
 \exists\,C\in\mR'\quad \text{with }\quad n(C)>1.
\end{equation}
Notice that if \eqref{eqn:c1} does not hold and $\gr\hq=\B(X,q)\sharp\k G$, then it follows that $\dim\hq=\infty$, provided that $|X|>1$, since
$\{(a_ia_j)^n\}_{n\in\N}$ is a linearly independent set in $\hq$.

\begin{lem}\label{lem:2simple}
Assume $X$ is indecomposable, and that \eqref{eqn:chiandrhd1} and \eqref{eqn:c1} hold.
Let $\rho,\mu\in\widehat{G_{ab}}$, and let $M$ be an $\hq$-module such that
$M=M[\rho]\oplus M[\mu]$, $\dim M[\rho]=\dim
M[\mu]=1$. Then $M$ is not simple.
\end{lem}

\medbreak

\begin{proof} Assume that there exists $M$ simple as in the hypothesis.
We first claim that $\rho\neq\mu$
and that, if $z\in M[\rho]$, then $a_i z\in M[\mu]$. In fact,
let $a_i z=u+w$ with $u\in M[\rho]$, $w\in M[\mu]$, then
$$
H_ta_i z=\rho(t)u+\mu(t)w, \quad \chi_i(t)a_{t\cdot i}H_t
z=\chi_i(t)\rho(t)a_{t\cdot i} z
$$
and taking $t=g_i$, we get $$
\rho(g_i)u+\mu(g_i)w=\chi_i(g_i)\rho(g_i)(u+w)
\overset{\eqref{eqn:chiandrhd1}}
= -\rho(g_i)u -\rho(g_i)w.$$
Thus $u=0$; hence $w\neq 0$ because $M$ is simple. Also, 
\begin{equation}\label{eqn:roimui}
\rho(g_i)=-\mu(g_i),\quad i\in X.
\end{equation}
By a symmetric
argument, $a_i (M[\mu]) = M[\rho]$.

\medbreak
Now, fix $0\neq z\in M[\rho]$,  $0\neq w\in M[\mu]$;
let $f_i$, $i\in X$, such that $a_i z=f_i w$. Then $(f_i)_{i\in X}$
satisfies \eqref{eqn:unob}, by \eqref{eqn:hq2}. As $X$ is
indecomposable and $M$ is simple, we have $f_i\neq 0$,
$\forall\,i\in X$. We necessarily have
\begin{equation}\label{eqn:pi}
 a_iw=p_iz, \quad
\text{for } \quad p_i=f_i^{-1}\lambda_i(1-\rho(g_i)^2).
\end{equation}
Note that $p_i\neq 0$ or otherwise $a_iw=0$,
$\forall\,i\in X$. As stated for $\{f_i\}$, the
family $\{p_i\}$ also satisfies
\eqref{eqn:unob}, with the roles of $\rho$ and $\mu$ interchanged.

\medbreak

Assume that there is $C\in\mR'$, with $n(C)>1$. We now show that this
contradicts the existence of $M$. Let $(i_2,i_1)=(i,j)\in C$, then
$$
a_C z=\sum_{h=1}^{n(C)}\eta_hf_{i_h}a_{i_{h+1}}w=\sum_{h=1}^{n(C)}\eta_hf_{i_h}
\frac{\lambda_{i_{h+1}}}{f_{i_{h+1}}}(1-\rho(g_{i_{h+1}})^2)z.
$$
Let $t\in G$ such that $i=t\cdot j$ and recall that $i_h=i_{h-1}\rhd
i_{h-2}$. Since $g_{s\cdot k}=g_sg_kg_s^{-1}$, then
$$\rho(g_{i_{h+1}})^2=\rho(g_{j})^2, \quad \forall\,h.$$ Now, by
\eqref{eqn:lambda2},
$\lambda_{i_h}=\lambda_{i_{h-1}\rhd
i_{h-2}}=\chi_{i_{h-2}}(g_{i_{h-1}})^{-2}\lambda_{i_{h-2}}$, then
\begin{equation*}
 \lambda_{i_h}=\begin{cases}
          \zeta_h(C)^{-2}\chi_j(t)^{-2}\lambda_j & \text{if } 2|h,\\
          \zeta_h(C)^{-2}\lambda_j & \text{if } 2|h+1.
         \end{cases}
\end{equation*}
Additionally, by \eqref{eqn:unob}, we have
\begin{equation}\label{eqn:fih}
 f_{i_h}=\begin{cases}
          \zeta_h(C)^{-1}\chi_j(t)^{-1}\dfrac{\mu(t)}{\rho(t)}f_j & \text{if
}
2|h,\\
          \zeta_h(C)^{-1}f_j & \text{if } 2|h+1,
         \end{cases}
\end{equation}
for every $h=1,\dots,n(C)$. Therefore, we have that:
\begin{equation}\label{eqn:llave1}
\eta_h(C)\lambda_{i_{h+1}}\dfrac{f_{i_h}}{f_{i_{h+1}}}=
\begin{cases}
\dfrac{\mu(t)}{\rho(t)}\chi_j(t)^{-1}\lambda_j & \text{if } 2|h,\\
\quad &\quad \\
\dfrac{\rho(t)}{\mu(t)}\chi_j(t)^{-1}\lambda_j & \text{if } 2|h+1.
\end{cases}
\end{equation}
Analogously, if we analyze the element $a_C w$, we get
\begin{equation}\label{eqn:llave2}
\eta_h(C)\lambda_{i_{h+1}}\dfrac{p_{i_h}}{p_{i_{h+1}}}=
\begin{cases}
\dfrac{\rho(t)}{\mu(t)}\chi_j(t)^{-1}\lambda_j & \text{if } 2|h,\\
\quad &\quad \\
\dfrac{\mu(t)}{\rho(t)}\chi_j(t)^{-1}\lambda_j & \text{if } 2|h+1.
\end{cases}
 \end{equation}
However, notice that, if $h>1$,
\begin{align*}
\eta_h(C)\lambda_{i_{h+1}}\dfrac{p_{i_h}}{p_{i_{h+1}}}&=\eta_h(C)\lambda_{i_{h+1
}}\dfrac{\lambda_{i_h}(1-\rho(g_{i_h})^2)f_{i_{h+1}}}{\lambda_{i_{h+1}}
(1-\rho(g_{i_{h+1}})^2)f_{i_h}}\\
&=-\eta_{h-1}(C)\chi_{i_{h-1}}(g_{i_h})\lambda_{i_h}\dfrac{f_{i_{h+1}}}{f_{i_h}}
\overset{\eqref{eqn:unob}}{=}-\eta_{h-1}(C)\lambda_{i_h}
\dfrac{f_{i_{h-1}}}{f_{i_h}}\dfrac{\mu(t)}{\rho(t)}\\
&\overset{\eqref{eqn:llave1}}{=}\begin{cases}
-\dfrac{\mu(t)^2}{\rho(t)^2}\chi_j(t)^{-1}\lambda_j & \text{if } 2|h-1,\\
\quad &\quad \\
-\chi_j(t)^{-1}\lambda_j & \text{if } 2|h.
\end{cases}
\end{align*}
And from this equality together with \eqref{eqn:llave2}, we get
\begin{equation}\label{eqn:rhomenosmu}
\rho(t)=-\mu(t), \quad \text{if}\quad (i,j)\in C,\quad t\cdot j=i.
\end{equation}
But, as $i\rhd i=i$, we have that $\mu(g_it)=-\rho(g_it)$ and also
\begin{align*}
\mu(g_it)&=\mu(g_i)\mu(t)\overset{\eqref{eqn:roimui}}{=}
-\rho(g_i)\mu(t)=\rho(g_i)\rho(t)=\rho(g_it),
\end{align*}
which is a contradiction.

\end{proof}

Assume $X$ is indecomposable. Next, we describe indecomposable modules which are
sums of two different isotypical components, provided
that \eqref{eqn:chiandrhd1} and \eqref{eqn:c1} hold.

\begin{theorem}\label{teo:1}
Let $\rho\neq\mu\in\widehat{G_{ab}}$. Assume $X$ is indecomposable and both 
\eqref{eqn:chiandrhd1} and \eqref{eqn:c1} hold. Let
$M=M[\rho]\oplus M[\mu]$ be an $\hq$-module, with $\dim M[\rho]$, $\dim
M[\mu]>0$. Then $M$ is not simple. 

Moreover, $M$ is a direct sum of
modules of the form $S_\rho^\gamma$, $S_\mu^\delta$,
$M_{\rho,\mu}^{\gamma',\delta'}$ and $M_{\mu, \rho,}^{\delta'',\gamma''}$ for
various $\gamma,\delta$, $\gamma',\delta'$, $\gamma'',\delta''$.
\end{theorem}
\begin{proof}
Take $0\neq z\in M[\rho]$. As in the first part of the proof of Lemma
\ref{lem:2simple}, it follows from \eqref{eqn:chiandrhd1}
that $\rho\neq\mu$ and that, if $0\neq z\in M[\rho]$, then $a_i z\in
M[\mu]$. Now, $a_i w=a_i^2
z=\lambda_i(1-\rho(g_i)^2)z$, and thus the space $\k\{z,w\}$ is
$a_i$-stable. As $X$ is indecomposable, it follows that this is a
submodule. Let $K=\ker a_i$. Here we see $a_i$ as an operator in $\End
M$. This subspace is $G$-stable: if $u\in
K$, $u=z+w$, with $z\in M[\rho]$, $w\in M[\mu]$, then
$0=a_iu=a_iz+a_iw\Rightarrow z,w\in K$, since $a_i w\in M[\rho]$,
$a_i z\in M[\mu]$. Thus $\rho(t)z=H_tz$ and $\mu(t)w=H_tw\in K$,
$\forall\,t\in G$. Therefore $G\cdot u\subset K$. The same holds for
$I=\img a_i$. Let $T$ be a $G$-submodule such that $M=K\oplus T$
(recall $\k G$ is semisimple). Let
$$
K=\ker a_i=K[\rho]\oplus K[\mu], \quad T=T[\rho]\oplus T[\mu], \quad
I=\img a_i=I[\rho]\oplus I[\mu].
$$
Notice that $K\neq 0$. In fact, if $K=0$, then the space
$\k\{z,w\}$ would be a simple 2-dimensional $\hq$-module, contradicting
Lemma \ref{lem:2simple}. Thus $K\neq 0$. Then $\gamma_i=0$, $\forall\,i\in X$
and
$a_i^2\cdot M=0$. Notice that in this case $I[\psi]\subseteq
K[\psi]$, for $\psi=\rho$ or $\mu$, and thus we have
$K[\psi]=I[\psi]\oplus J[\psi]$. As $G$-modules, we have $$
M_{|G}\cong
\bigoplus\limits_{\psi=\rho,\mu}M[\psi]=\bigoplus\limits_{\psi=\rho,\mu}I[\psi]
\oplus
J[\psi]\oplus T[\psi],$$ and this induces the following decomposition
of $\hq$-modules:
$$
M\cong J[\rho]\oplus J[\mu]\oplus (I[\rho]+T[\mu])\oplus
(I[\mu]+T[\rho]).
$$
Let $\psi=\rho$ or $\mu$. If $J[\psi]\neq 0$, then
\eqref{eqn:tres} holds for $\psi$, and $J[\psi]$ is a sum of 1-dimensional
$\hq$-modules, by Prop. \ref{pro:dims1}. Let
$\{w_1,\dots,w_k\}$ be a basis of $T[\mu]$. Then
$\{a_iw_1,\dots,a_iw_k\}$ is a basis of $I[\rho]$. In fact, if $z\in
I[\rho]$, $z=a_iw$, $w\in T[\mu]$, there are
$\sigma_1,\dots,\sigma_k\in\k$ such that $ w=\sum_{j=1}^k\sigma_jw_j
\ \text{and then } \ z=\sum_{j=1}^k\sigma_ja_iw_j. $ If, on the
other hand, $\{\sigma_j\}_{j=1}^k\in\k$ satisfy
$0=\sum_{j=1}^k\sigma_ja_iw_j$ then $\sum_{j=1}^k\sigma_jw_j\in
K[\mu]$, and as $K\cap T=0$, $\sigma_j=0\,\forall\,j=1,\dots,k$.
Thus $I[\rho]+T[\mu]=\bigoplus_{j=1}^k\langle w_j\rangle$ as
$\hq$-modules. By Lemma \ref{lem:wrhomu}, for each $j=1,\dots,k$ there exists
$\delta_j,\gamma_j\in\k^{*X}$ such that $\langle w_j\rangle\cong
M_{\mu,\rho}^{\delta_j,\gamma_j}$. A similar statement follows for
$I[\mu]+T[\rho]$. Therefore, there are $m_\rho,m_\mu$,
$m_{\rho,\mu},m_{\mu,\rho}\in\mathbb{N}_0$,
$\{\xi_j\}_{j=1}^{m_\rho},\{\pi_j\}_{j=1}^{m_\mu},\{\delta_j\}_{j=1}^{m_{\rho,
\mu
}},\{\gamma_j\}_{j=1}^{m_{\rho,\mu}},\{\sigma_j\}_{j=1}^{m_{\mu,\rho}},$
$\{\tau_j\}_{j=1}^{m_{\mu,\rho}}\in\k^X$ such that
$$
M\cong\bigoplus_{j=1}^{m_\rho} S_\rho^{\xi_j}\oplus \bigoplus_{j=1}^{m_\mu}
S_\rho^{\pi_j}\oplus \bigoplus_{j=1}^{m_{\rho,\mu}}
M_{\mu,\rho}^{\delta_j,\gamma_j}\oplus \bigoplus_{j=1}^{m_{\mu,\rho}}
M_{\mu,\rho}^{\sigma_j,\tau_j},
$$
where $m_\rho$ (resp. $m_\mu$) is non-zero only if \eqref{eqn:tres} holds for
$\rho$ (resp. $\mu$), $\xi_j$, $\pi_j$ and satisfy \eqref{eqn:uno} and
\eqref{eqn:dos} for $\rho,\mu$ respectively. On the other hand,
$m_{\rho,\mu}\neq 0$ only if \eqref{eqn:unob} holds for $\rho,\mu$ and
$\delta_j,\gamma_j$ satisfy \eqref{eqn:dosb}. Similarly for $m_{\mu,\rho}$,
$\sigma_j,\tau_j$.
\end{proof}

\subsection{The case $G=\s_n$, $n\geq 3$}

\

Let $\Lambda,\Gamma,\lambda\in\k$,
$t=(\Lambda,\Gamma)$,
$\iota:\mO_2^n\hookrightarrow\s_n$ the inclusion, $\cdot:\s_n\times
X\to X$ the action given by conjugation, $-1$ the constant cocycle
$q\equiv -1$ and $\chi$ the cocycle given by, if $\tau, \sigma\in \mO_2^n$,
$\tau=(ij)$ and $i<j$:
\begin{align*} & \chi(\sigma, \tau) =
\begin{cases}
  1,  & \mbox{if} \ \sigma(i)<\sigma(j) \\
  -1, & \mbox{if} \ \sigma(i)>\sigma(j),
\end{cases} & & \text{see \cite[Ex. 5.3]{MS}.}
\end{align*}
Then the ql-data:

\begin{itemize}
 \medbreak\item
    $\mQ_n^{-1}[t]=(\s_n,\mO_2^n,-1,\cdot,\iota,\{0,\Lambda,\Gamma\})$,
    $n\geq 4$;

 \medbreak\item
    $\mQ_n^{\chi}[\lambda]=(\s_n,\mO_2^n,\chi,\cdot,\iota,\{0,0,\lambda\})$,
$n\geq 4$;

 \medbreak\item
$\mQ_3^{-1}[\lambda]=(\s_3,\mO_2^3,-1,\cdot,\iota,\{0,\lambda\})$;
\end{itemize}
define pointed Hopf algebras over $\s_n$, for $n$ as appropriate,
\cite{AG2, GG}.
\begin{rem}\label{rem:a}
Notice that the racks $\mO_2^n$, $n\geq 3$ are indecomposable and
that \eqref{eqn:chiandrhd1} is satisfied for both cocycles. In this
case, $\widehat{G_{ab}}=\{\eps,\sg\}$, where $\eps$, resp. $\sg$, stands for the
trivial, resp. sign, representation. In any case, \eqref{eqn:gigj} holds.
Bear also in mind that
$\s_n=\langle\mO_2^n\rangle$. In this case, the function $\ell:G\to\Z$ is
well-known and $\psi:G\to\{\pm 1\}\subset\k^*$ coincides with the
sign function, by \eqref{eqn:chiandrhd1}. Moreover, \eqref{eqn:c1} holds in
all of these ql-data.
\end{rem}
\begin{pro}\label{pro:1sn}
Let $A=\mH(\mQ_n^{-1}[t])$ or $\mH(\mQ_3^{-1}[\lambda])$. Let $M$ be
an $A$-module such that $M_{|\s_n}=M[\eps]\oplus M[\sg]$, $\dim M[\eps]=p$,
$\dim M[\sg]=q$. Then
\begin{enumerate}
  \item $M$ is simple if and only if $M=S_\eps$ or  $M=S_{\sg}$.

  \item $M$ is indecomposable if and only if $M$ is simple or $p=q=1$. In this
last case, there are two non-isomorphic indecomposable modules, namely
  $M_{\eps,\sg}$ and $M_{\sg,\eps}$.
\end{enumerate}
\end{pro}
\begin{proof}
It follows by
Props. \ref{pro:dim1} and \ref{pro:dims1}, and by Lemma \ref{lem:2simple} that
$S_\eps$ and $S_{\sg}$ are the unique two simple modules. The second item
follows by Thm. \ref{teo:1} and Lemma \ref{lem:wrhomu2}.
\end{proof}

\begin{pro}
Let  $n\geq 4$. Let $M$ be a
$\mH(\mQ_n^{\chi}[\lambda])$-module such that $M_{|\s_n}=M[\eps]\oplus
M[\sg]$, with $\dim M[\eps]=p$, $\dim
M[\eps]=q$, $p,q\geq 0$. Then
$M$  is indecomposable if and only if it is simple if and only if $M=S_\eps$ or $M=S_{\sg}$.

\end{pro}
\begin{proof}
The determination of the simple modules follows from Props. \ref{pro:dim1} and \ref{pro:dims1} and Lemma
\ref{lem:2simple}. By Lemma
\ref{lem:wrhomu2} there are no extensions between 1-dimensional modules.
Hence, the Prop. follows from Thm.
\ref{teo:1}.
\end{proof}

\section{General facts}

Let $H$ be a Hopf algebra, $V\in\ydh$. The Nichols algebra $\B(V)=\oplus_{n\geq 0}\B^n(V)$ is a
graded braided Hopf algebra in $\ydh$ generated by $V$, in such a way that
$V=\B^1(V)=\mP(\B(V))$, that is, it is generated in degree one by its primitive
elements which in turn coincide with the module $V$. This algebra is uniquely
determined, up to isomorphism. See \cite{AS} for details.

\medbreak

Let $G$ be a finite group. Let $X$ be a rack, $q$ a 2-cocycle and assume
that there
exists a YD-realization of $(X,q)$ over $G$. We denote by $\B(X,q)$ the
corresponding Nichols algebra.

\subsection{Simple modules over bosonizations}

\

Consider the bosonization $\mA=\B(X,q)\sharp\k G$.
As an algebra, $\mA$ is generated by
$\B(X,q)$ and $\k G$; the product is defined by $(a\sharp
t)(b\sharp s)=a(t\cdot b)\sharp ts$, here $\cdot$ stands for the action in
$_G^G\yd$. See \cite[2.5]{AS} for details. In what follows, we shall assume that
$\B(X,q)$, and thus $\mA$, is finite dimensional.
The following proposition is well-known. We state it and prove it here for
the sake of completeness.

\begin{pro}\label{pro:bxqsimple}
The simple modules for $\mA$ are in bijective correspondence
with the simple modules over $G$: Given $\rho\in\widehat G$, $S_\rho$ is the
$\mA$-module such that
$$
S_\rho\cong W_\rho \text{ as } G\text{-modules,} \quad\text{and  }\quad  a_iS_\rho=0, \quad \forall\,i\in X .
$$
This correspondence preserves tensor products and duals.
\end{pro}
\begin{proof}
With the action stated above, it is clear that for each $\rho\in\widehat G$,
$S_\rho$ is an $\mA$-module. If
$\B(X,q)^+$ denotes the maximal graded ideal of $\B(X,q)$, then the
Jacobson radical $J=J(\mA)$ is given by $J=\B(X,q)^+\sharp\k G$. In fact $J$ is
a maximal nilpotent ideal (since $\mA$ is graded and finite dimensional) and
$\mA/J\cong \k G$ is semisimple. This also shows that the list
$\{S_\rho:\rho\in\widehat G\}$ is an exhaustive list of $\B(X,q)$-modules, which
are obviously
pairwise non-isomorphic. The last assertion follows since $a_i
\left(S_\rho\ot S_\mu\right)=0$ and $\mathcal{S}(a_i)=-H_{g_i}^{-1}a_i$.
\end{proof}

\subsection{Projective covers of modules over quadratic liftings}

\

Let $B$ be a ring, $M$ a left $B$-module. A \emph{projective cover} of
$M$ is a pair $(P(M),f)$ with $P=P(M)$ a projective $B$-module and
$f:P\to M$ an\emph{ essential map}, that is $f$ is surjective and
for every
 $N\subset M$ proper submodule, $f(N)\neq M$. We will not explicit the map $f$
when it is obvious. Projective covers
are unique up to isomorphism and
always exist for finite-dimensional $\k$-algebras, see \cite[Sect.
6]{CR}. Moreover,
 \begin{equation}\label{eqn:proj}
_BB\cong \bigoplus\limits_{S\in\widehat{B}}P(S)^{\dim S}.
\end{equation}

Fix $G$ a finite group and $H$ a pointed Hopf algebra over $G$. Let $\{e_i\}_{i=1}^N$ be a complete set of orthogonal primitive idempotents for $G$ and set $I_j=He_j$, for $1\leq j\leq N$. 

\begin{lem}
$I_j=\ind_{\k G}^H\k Ge_j$. In particular, if $\k Ge_j\cong \k Ge_h$ as $G$-modules, then $I_j\cong I_h$ as $H$-modules. 

Moreover, $H\cong\bigoplus_{\rho\in\widehat{G}}I_\rho^{\dim \rho}$ as
$H$-modules, where $I_\rho=\ind_{\k G}^H W_\rho$, and thus $I_\rho$ is a
projective $H$-module.
\end{lem}
\begin{proof}
Let $\psi:\ind_{\k G}^H\k Ge_j\to H$ be the composition of the multiplication
$m:H\ot_{\k G}\k G\to H$ with the inclusion $H\ot_{\k G}\k Ge_j\to H\ot_{\k G}\k
G$. It follows that $\img\psi=I_j$. Then $I_j=\ind_{\k G}^H\k Ge_j$ and $I_j$
does not depend on the idempotent $e_j$ but on the simple module $W_\rho=\k
Ge_j$. Therefore, as $\k G=\oplus_{i=1}^N \k G e_i$, we have that
$H\cong\bigoplus_{\rho\in\widehat{G}}I_\rho^{\dim \rho}$.
\end{proof}

Let $\{H_n\}_{n\in\N}$ be the coradical filtration of $H$, $$\gr^nH=H_n/H_{n-1}, \quad 
\gr H=\oplus_{n\geq 0} \gr^nH.$$ We know that there exists $R\in\ydg$ such that
$\gr H\cong R\sharp \k G$, see \cite[2.7]{AS}. Let $\pi_n:H_n\to \gr^nH$
be the canonical projection. As every $H_n$ is $\ad(G)$-stable, it follows that
$\pi_n$ is a morphism of $G$-modules. Therefore there exists a section
$\gr^nH\to H_n$ and $H_n\cong \gr^nH\oplus H_{n-1}$ as $G$-modules. By an
inductive argument we have that $H_n\cong \gr^nH\oplus
\gr^{n-1} H\oplus\dots\oplus \gr^0 H$. And thus it follows that $H\cong \gr H$
as $G$-modules. Moreover, it follows that, if we consider the adjoint action on
$\k G$, $\gr H\cong R\ot\k G$ as $G$-modules, via the diagonal action. Thus,
$H\cong R\ot\k G$ as $G$-modules.

\begin{pro}\label{pro:proj}
Let $\gr H=R\sharp\k G$.
\begin{enumerate}
 \item $I_\eps\cong R$ as $G$-modules.
\item Assume there exists a simple $H$-module $M$ such that $M_{|\k G}$ is a
simple $G$-module $W_\rho$. Then $P(M)$ is a direct summand of $I_\rho$. In
particular, if $I_\rho$ is indecomposable, then $I_\rho\cong P(M)$.
 \item If $H=R\sharp\k G$, $I_\rho$ is the projective cover of $S_\rho$, see
Prop. \ref{pro:bxqsimple}.
 \end{enumerate}
 \end{pro}
\begin{proof}
Let $W_\eps$ be the trivial $G$-module. Since $I_\eps=\ind_{\k G}^H W_\eps$ and $H\cong R\ot \k G$, we have
$$
(I_\eps)_{|G}\cong ((R\ot\k G)\ot_{\k G} W_\eps)_{|G}\cong R_{|G}.
$$
Thus the first item follows. Let now $M$ be an $H$-module such that $M_{|\k
G}=W_\rho$. If $(P(M),f)$ is the projective cover of
$M$, we have the commutative diagram:
\begin{equation*}
 \xymatrix{
  & & I_\rho\ar@{->>}[d]^\pi\ar@{-->}[lld]_{\tau}\\
  P(M)\ar@{->>}[rr]^{f} &  & M}
\end{equation*}
where $\pi:I_\rho\to M$ is the factorization of the action $\cdot:H\ot M\to M$
through $H\ot M\twoheadrightarrow I_\rho=H\ot_{\k G} W_\rho$.
As $f(\tau(I_\rho))=\pi(I_\rho)=M$ and $f$ is essential, we have
an epimorphism $I_\rho\twoheadrightarrow P(M)$ and $P(M)$ is a direct summand of
$I_\rho$. Thus $I_\rho\cong P(M)$,
if $I_\rho$ is assumed to be indecomposable.

Finally, assume $H=R\sharp\k G$. If $P(S_\rho)$ is
the projective cover of $S_\rho$, we must have $\dim P(S_\rho)\leq \dim
I_\rho=\dim R\dim W_\rho$. But we
see that this is in fact an equality from the formulas:
\begin{align*}
&\dim H=\dim R\sum_{\rho\in\widehat G}\dim
W_\rho^2=\sum_{\rho\in  \widehat G}(\dim R\dim W_\rho)\dim W_\rho\\
& \dim H=\sum_{\rho\in  \widehat G}\dim P(S_\rho)\dim S_\rho=\sum_{\rho\in 
\widehat
G}\dim P(S_\rho)\dim W_\rho.
\end{align*}
\end{proof}

\subsection{Representation type}

\

We comment on some general facts about the representation
type of a finite dimensional algebra, that will be employed in \ref{subsub:reprtypea0} and \ref{subsub:reprtypea1}. Let $B$ be a finite dimensional $\k$-algebra, $\widehat{B}=\{S_1,\ldots,S_n\}$ a
complete list of non-isomorphic simple $B$-modules. The \emph{Ext-Quiver} (also
\emph{Gabriel quiver}) of $B$ is the quiver $ExtQ(B)$ with vertices
$\{1,\ldots,n\}$
and $\dim\ext^1_B(S_i,S_j)$ arrows from the vertex $i$ to the vertex
$j$. Then $B$ is Morita equivalent to the basic algebra $\k
ExtQ(B)/I(B)$, where $\k ExtQ(B)$ is the path algebra of the quiver $ExtQ(B)$
and $I(B)$ is an ideal contained in the bi-ideal of paths of length
greater than one. Recall that for any two $B$ modules $M_1,M_2$
there is an isomorphism of abelian groups
$$
\ext_B^1(M_1,M_2)=\{\text{equivalence classes of extensions of
}M_1\text{ by }M_2\},
$$
where the element $0$ is given by the trivial extension $M_1\oplus
M_2$.

Given a quiver $Q$ with vertices $V=\{1,\ldots, n\}$, its
\emph{separation diagram} is the unoriented graph with vertices
$\{1',\ldots,n',1'',\ldots, n''\}$ and with an edge $i'$---$j''$ for
each arrow $i\to j$ in $Q$. If $B$ is algebra, we speak of the
separation diagram of $B$ referring to the separation diagram of its
Ext-Quiver.
\begin{theorem}\cite[Th. 2.6]{ARS}\label{teo:aus}
  Let $B$ be an Artin algebra with radical square zero. Then $B$ is
  of finite (tame) representation type if and only if its separated
  diagram is a disjoint union of finite (affine) Dynkin
  diagrams.\qed
\end{theorem}
\begin{lem}\label{lem:extq}
  Let $J$ be the radical of $B$. Then $ExtQ(B)=ExtQ(B/J^2)$.
\end{lem}
\begin{proof}
First, it is immediate that $\widehat{B}=\widehat{B/J^2}$. Let
$S,T\in\widehat{B}$. As any $B/J^2$-module is a $B$-module, we have
$\ext_{B/J^2}^1(S,T)\subseteq\ext_B^1(S,T)$. Now, let
$$
0\to T\hookrightarrow V\twoheadrightarrow S\to 0\in B-\mo,\quad x\in
V, a_1,a_2,\in J.
$$
If $x\in T\subset V$, then $a_1x=0\Rightarrow a_2a_1x=0$. If
$x\notin T$, then $0\neq\bar{x}\in V/T\cong S$ and thus
$a_1\bar{x}=0$, that is $a_1x\in T$, and therefore $a_2a_1x=0$.
Thus, the above exact sequence in $B-\mo$ gives rise to an exact
sequence in $B/J^2-\mo$, proving the lemma.
\end{proof}

\section{Representation theory of pointed Hopf algebras over $\s_3$}

In this Section we investigate the representations of the finite
dimensional pointed Hopf algebras over $\s_3$. We will denote by
$\mA_\lambda$, $\lambda\in\k$, the algebra $\mH((\mQ_3^{-1}[\lambda]))$. This
algebra
was introduced in \cite{AG1}. Explicitly,
it is generated by elements $H_t$, $a_i$,
$t,i\in\mO_2^3$; with relations
\begin{align*}
H_tH_sH_t&=H_sH_tH_s, \ H_t^2=1, & s\neq t\in \mO_2^3;\\
H_ta_i &= -a_{t\sigma i}H_t, & t,i\in\mO_2^3;\\
 a_{12}^2&=0, & \\
a_{12}a_{23}+a_{23}a_{13}+a_{13}a_{12}&=\lambda(1-H_{12}H_{23}).
\end{align*}
$\mA_\lambda$ is a Hopf algebra of dimension $72$. If $H$ is a
finite-dimensional pointed Hopf algebra with $G(H)\cong\s_3$, then
either $H\cong\k\s_3$, $H\cong\mA_0$ or $H\cong\mA_1$ \cite[Theorem
4.5]{AHS}, together with \cite{MS,AG1,AZ}.

We will determine all simple modules over $\mA_0$ and $\mA_1$, along with
their projective covers and fusion rules.
We will also show that these algebras are not of finite representation type and
classify indecomposable modules satisfying certain restrictions. 
\begin{rem}\label{rem:Ht y a12}
Notice that to describe an $\mA_\lambda$-module supported on a given $G$-module, it is enough to describe the action of $a_{12}$, since $a_{13}, a_{23}\in\ad(G)(a_{12})$.
\end{rem}

\subsection{Simple $\k\s_3$-modules}

\
We will need some facts about the representation theory of $\s_3$, which we state next. Besides the modules $W_\eps$ and $W_{\sg}$ associated to the characters
$\eps$ and $\sg$, respectively, there is
one more simple $\k\s_3$-module, namely the standard representation
$W_{\st}$. This module has dimension 2. We fix $\{v,w\}$ as its
canonical basis. In this basis the representation is given by the
following matrices:
\begin{equation*} [H_{12}]=\begin{pmatrix} 0&1\\1&0\end{pmatrix}, \quad
[H_{23}]=\begin{pmatrix} 1&0\\-1&-1\end{pmatrix}, \quad
[H_{13}]=\begin{pmatrix} -1&-1\\0&1\end{pmatrix}.
\end{equation*}
Given a $\k\s_3$-module $W$, we denote by $W[\st]$ the isotypical component
corresponding to this representation.

\subsection{Representation theory of $\mA_0$}

\

\begin{pro}
There are exactly three simple $\mA_0$-modules, namely the extensions $S_\eps$,
$S_{\sg}$ and $S_{\st}$ of the simple $\k\s_3$-modules.
\end{pro}
\begin{proof}
Follows from Prop. \ref{pro:bxqsimple}.
\end{proof}

\subsubsection{Some indecomposable $\mA_0$-modules}

\

Fix $\langle x\rangle_{\s_3}=W_\eps$, $\langle y\rangle_{\s_3}=W_{\sg}$,
$\langle v,w\rangle_{\s_3}=W_{\st}$. 

\begin{lem}\label{lem:indesc0}
There are exactly four non-isomorphic non-simple indecomposable
$\mA_0$-modules  of dimension 3:
\begin{align}
\tag{i} &M_{\st,\eps}=\k\{x,v,w\}, &&\text{with}&& a_{12}\cdot v=x,
&&a_{12}\cdot x=0;\\
\tag{ii} &M_{\st,\sg}=\k\{y,v,w\}, &&\text{with}&& a_{12}\cdot v=y,&& a_{12}\cdot y=0;\\
\tag{iii} &M_{\eps,\st}=\k\{x,v,w\}, &&\text{with}&& a_{12}\cdot x=v-w,&& a_{12}\cdot v=0;\\
\tag{iv} &M_{\sg,\st}=\k\{y,v,w\}, &&\text{with}&& a_{12}\cdot y=v+w,&& a_{12}\cdot v=0.
\end{align}

In particular, $\dim\ext_{\mA_0}^1(S_{\st},S_\sigma)=\dim\ext_{\mA_0}^1(S_\sigma,S_{\st})=1$,
$\sigma\in\{\eps,\sg\}$.
\end{lem}
\begin{proof}
By Prop.
\ref{pro:1sn}, we know that such an $\mA_0$-module $M$ must contain a copy of $W_{\st}$. Thus
$M_{|\s_3}\cong W_{\eps}\oplus W_{\st}$ or
$M_{|\s_3}\cong W_{\sg}\oplus W_{\st}$. The lemma now follows by
straightforward computations.
\end{proof}

\begin{pro}\label{pro:suma de st0}
The non-isomorphic indecomposable modules which are extensions of
$S_{\st}$ by itself are indexed by $\mathbb{P}_\k^1$. In particular, it follows
that $\dim\ext_{\mA_0}^1(S_{\st},S_{\st})=1$.
\end{pro}
\begin{proof}
If $\{v_1,v_2,w_1,w_2\}$ is basis of such a module, with $\{v_2,
w_2\}_{|\s_3}=W_{\st}$, $\{v_1,w_1\}\cong M_{st}$, then a necessary
condition is that $a_{12}v_2=av_1+bw_1$, $a\neq 0$ or $b\neq 0$. It is easy to
see that this formula defines in fact an indecomposable $\mA_0$ module
$M_{(a,b)}$ for each $(a,b)$ and that two of these modules, $M_{(a,b)}$ and
$M_{(a',b')}$, are
isomorphic if and only if $\exists\,\gamma\neq 0$ such that
$(a,b)=\gamma (a',b')$.
\end{proof}

\subsubsection{Representation type of $\mA_0$}\label{subsub:reprtypea0}

\begin{pro}
 $\mA_0$ is of wild representation type.
\end{pro}
\begin{proof}
From Lemmas \ref{lem:wrhomu2} and
\ref{lem:indesc0} together with Prop. \ref{pro:suma de st0}, we see that the Ext-Quiver of $\mA_0$ is
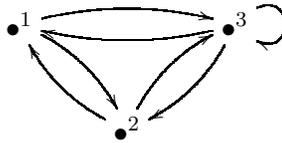
\begin{equation*}
 \xymatrix{\bullet^1\ar@/^/[rr]\ar@/^/[dr]&
&\bullet^3\ar@/^/[ll]\ar@/^/[dl]\ar@(ur,dr)\\
  &\bullet^2\ar@/^/[ul]\ar@/^/[ur]&}
\end{equation*}
where we have ordered the simple modules as
$\{S_\eps,S_{\sg},S_{\st}\}=\{1,2,3\}$.
Thus, the separation diagram of $\mA_0$ is
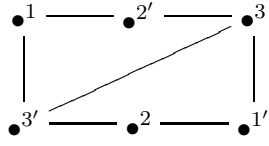
\begin{equation*}
\xymatrix{\bullet^1\ar@{-}[r]&\bullet^{2'}\ar@{-}[r]&\bullet^{3}\ar@{-}[d]\\
\bullet^{3'}\ar@{-}[u]\ar@{-}[urr]&\bullet^{2}\ar@{-}[l]&\bullet^
{1'}\ar@{-}[l]}
\end{equation*}
which implies that $\mA_0$ is wild.
\end{proof}

\subsection{Representation theory of  $\mA_1$}

\

We investigate now the simple modules of $\mA_1$, their fusion rules and
projective covers, and also the representation type of this algebra.

\subsubsection{Modules that are sums of 2-dimensional representations}

We first focus our
attention on those $\mA_1$-modules supported on sums of standard
representations of $\k\s_3$.

\begin{lem}\label{lem:Sst(i)}
Let $M_{\st}=\k\{v,w\}$. Then, the following formulas define four non-isomorphic
$\mA_1$-modules supported on $M_{\st}$:
\begin{align}
\tag{i} &a_{12}v=\im(v-w),&& a_{12}w=\im(v-w);\\
\tag{ii} &a_{12}v=-\im(v-w),&& a_{12}w=-\im(v-w);\\
\tag{iii} &a_{12}v=\frac\im3(v+w),&& a_{12}w=-\frac\im3(v+w);\\
\tag{iv} &a_{12}v=-\frac\im3(v+w),&& a_{12}w=\frac\im3(v+w).
\end{align}
They are simple modules, and we denote them by $S_{\st}(\im),
S_{\st}(-\im), S_{\st}(\frac\im3)$, $S_{\st}(-\frac\im3)$, respectively.
\end{lem}
\begin{proof}
Straightforward.
\end{proof}

\begin{pro}\label{pro:suma de st}
Let $p\in\mathbb{N}$ and let $M$ be an $\mA_1$-module such that
$M=M[\st]$, $\dim M=2p$. Then $M$ is completely reducible. 

$M$ is simple if only if $p=1$. In this case, it is isomorphic to one of the modules $S_{\st}(\im),
S_{\st}(-\im), S_{\st}(\frac\im3), S_{\st}(-\frac\im3)$.
\end{pro}
\begin{proof}
Let $\{v_i, w_i\}_{i=1}^p$ be copies of the canonical basis of
$W_{\st}$ such that $\{v_i, w_i\}_{i=1}^p$ is a linear basis of $M$. Let
$v=(v_1, \ldots, v_p)$, $w=(w_1, \ldots,
w_p)$. Now,
there must exist matrices $\alpha,\beta\in\k^{p\times p}$ such that
$a_{12}\cdot v=\alpha v+\beta w$ and thus $a_{12}\cdot w=-\beta
v-\alpha w$, by acting with $H_{12}$. By acting with the rest of the
elements $H_t$ we get:
\begin{align*}
a_{13}\cdot v&=-(\alpha+\beta) v+2(\alpha+\beta) w, && a_{13}\cdot
w=-\beta v+(\alpha+\beta) w,
\\
a_{23}\cdot v&=-(\alpha+\beta) v+\beta w && a_{23}\cdot
w=-2(\alpha+\beta) v+(\alpha+\beta) w.
\end{align*}
Now, $0=a_{12}^2 v=\alpha a_{12}\cdot v+\beta a_{12}\cdot
w=(\alpha^2-\beta^2)v+(\alpha\beta-\beta\alpha)w$, and this implies
that $\alpha^2=\beta^2$, $\alpha\beta=\beta\alpha$. Hence,
\begin{align*}
  (a_{12}a_{13}+a_{13}a_{23}+a_{23}a_{12})\cdot
v&=(-5\alpha^2-4\alpha\beta)(v+w),\quad\\
\text{while}\quad (1-H_{12}H_{13})\cdot v&=v+w,
\end{align*}
and thus $-5\alpha^2-4\alpha\beta=\id$.

Now, we have that, in particular,
$-5\alpha-4\beta=\alpha^{-1}$ and therefore
$\beta=-\frac54\alpha-\frac14\alpha^{-1}$. Thus,
\begin{equation*}
\alpha^2=\beta^2=\frac1{16}(5\alpha+\alpha^{-1})^2=\frac1{16}(25\alpha^2+\alpha^
{-2}+10\id),
\end{equation*}
from where it follows $(\alpha^2)^{-1}=-9\alpha^2-10\id$ and
$\id=-9\alpha^4-10\alpha^2$, which is equivalent to
\begin{equation}\label{eqn:alpha}
(\alpha^2+\frac59\id)^2=\frac{16}{81}\id.
\end{equation}
This gives, in particular, that if $\theta\in\k$ is an eigenvalue of
$\alpha$, then $\theta\in
L(\alpha):=\{\pm\mathrm{i},\pm\frac{\mathrm{i}}3\}$. Now, let
$\alpha\in\k^{p\times p}$ be a matrix satisfying equation
\eqref{eqn:alpha}. A simple analysis of the possible Jordan forms
$J(\alpha)$ of $\alpha$ gives
$J(\alpha)=\diag(\theta_1,\ldots,\theta_p)$, for some $\theta_i\in
L(\alpha)$, $i=1,\ldots,p$. If $p>1$, we
get that there is a basis of $M$ in which $\alpha$ (and consequently
$\beta$) is a diagonal matrix, and so $M$ is completely reducible.

On the other hand, if $p=1$, $\alpha\in L(\alpha)$ and
$\beta=\pm\alpha$ give the module structures defined in Lemma
\ref{lem:Sst(i)}. 
\end{proof}

\subsubsection{Classification of simple modules over $\mA_1$}

Now, we present the classification of all simple $\mA_1$-modules.

\begin{theorem}\label{teo:simple s3}
Let $M$ be a simple $\mA_1$-module. Then $M$ is isomorphic to one and only one of
the following:\begin{itemize}
\item $S_\eps$;
\item $S_{\sg}$;
\item $S_{\st}(\mathrm{i})$, $S_{\st}(-\mathrm{i})$,
$S_{\st}(\frac{\mathrm{i}}{3})$ or $S_{\st}(-\frac{\mathrm{i}}{3})$.
\end{itemize}
\end{theorem}
\begin{proof}
We know that the listed modules are all simple. In view of
Props. \ref{pro:1sn} and \ref{pro:suma de st}, we are left to deal
with the case in which $M_{|\s_3}= M[\eps]\oplus
M[\sg]\oplus M[\st]$, with $\dim M[\eps]=n$, $\dim
M[\sg]=m$, $\dim M[\st]=p$, $n+m, p> 0$. 
Let $\{x_1,\ldots, x_n, y_1, \ldots, y_m, v_1, \ldots, v_p, w_1,\ldots, w_p\}$
be a basis of $M$ such that
$\k\{x_i\}\cong W_{\eps}$, $i=1, \ldots, n$, $\k\{y_j\}\cong
W_{\sg}$, $j=1, \ldots, m$, $\k\{v_k,w_k\}\cong
W_{\st}$, $k=1, \ldots, p$. Using the action of $H_{12}$, we find
that there are matrices $\alpha\in\k^{n\times m}$,
$\beta\in\k^{n\times p}$, $\gamma\in\k^{m\times n}$,
$\eta\in\k^{m\times p}$, $a\in\k^{p\times n}$, $b\in\k^{p\times m}$
and $c, d\in\k^{p\times p}$, such that, if $x=(x_1, \ldots, x_n)$,
$y=(y_1, \ldots, y_m)$, $v=(v_1, \ldots, v_p)$, $w=(w_1, \ldots,
w_p)$, the action of $a_{12}$ is determined by the following
equations:
\begin{align*}
  a_{12}\cdot x&=\alpha y+\beta(v-w), && a_{12}\cdot y=\gamma
x+\eta(v+w)\\
a_{12}\cdot v&=ax+by+cv+dw, &&  a_{12}\cdot w=-ax+by-dv-cw.
\end{align*}
We deduce as in Prop. \ref{pro:suma de st} the action of every $a_\sigma$:
\begin{align*}
  & a_{13}\cdot x=\alpha y-\beta v, & &  a_{13}\cdot v=-2ax-(c+d)v+2(c+d)w,\\
  & a_{13}\cdot y=\gamma x+\eta(v-2w) & & a_{13}\cdot w=-ax-by-dv+(c+d)w,\\
& a_{23}\cdot x=\alpha y+\beta w,& & a_{23}\cdot v=ax-by-(c+d)v+dw, \\
 &a_{23}\cdot y=\gamma x+\eta(w-2v), & & a_{23}\cdot w=2ax-2(c+d)v+(c+d)w.\\
\end{align*}
Recall that it is enough to find a subspace stable under the action of $a_{12}$
and the elements $H_t$, by Rem. \ref{rem:Ht y a12}. Now,

\begin{align*}
  0&=a_{12}^2x=(\alpha\gamma+2\beta
a)x+(\alpha\eta+\beta(c+d))(v+w);\\
  0&=a_{12}^2y=(\gamma\alpha+2\eta
b)y+(\gamma\beta+\eta(c-d))(v-w);\\
  0&=a_{12}^2v= (b\gamma+(c-d)a)x+(a\alpha+(c+d)b)y\\
  &\quad+(a\beta+b\eta+c^2-d^2)v+(-a\beta+b\eta+cd-dc)w;
\end{align*}

  \begin{align*}
 0&=(a_{12}a_{13}+a_{13}a_{23}+a_{23}a_{12})\cdot x=(3\alpha \gamma-3\beta
a)x-3\beta
by;\\ 0&=(a_{12}a_{13}+a_{13}a_{23}+a_{23}a_{12})\cdot y=9\eta
ax+3(\gamma\alpha-\eta
  b)y;\\
 v+w&=(a_{12}a_{13}+a_{13}a_{23}+a_{23}a_{12})\cdot v\\&=(-3a\beta
-3b\eta-c^2-4d^2-2dc-2cd)v\\
& \quad+(3a\beta+3b\eta-4c^2-d^2-2dc-2cd)w.
\end{align*}

Then we have the following equalities:
\begin{equation}\label{eqn:sistema1}
  \begin{cases}
  &0= \gamma\alpha=\alpha\gamma=\beta a=\beta b=\eta a=\eta b,\\
  &\beta(c+d)+\alpha\eta=0=\eta(c-d)+\gamma\beta,\\
  &  b\gamma+(c-d)a=0=a\alpha +(c+d)b,\\
  &d^2-c^2=a\beta+b\eta,\quad cd-dc=a\beta-b\eta \\
    &    3a\beta+ 3b\eta=-c^2-4d^2-2dc-2cd-\id\\
    &3a\beta+ 3b\eta=4c^2+d^2+2dc+2cd+\id.
  \end{cases}
\end{equation}
From the last two equations:
$$
   c^2-d^2=2(a\beta+ b\eta),\quad 5(c^2+d^2)+4(dc+cd)=-2\id,
  $$
and thus $a\beta+ b\eta=0$, $c^2=d^2$. Notice that the matrix of $a_{12}$ in the
chosen basis is:
\begin{equation*}
    [a_{12}]=
  \begin{pmatrix}
0 & \,^t\gamma & \,^ta&-\,^ta\\
\,^t\alpha& 0& \,^tb&\,^tb\\
\,^t\beta&\,^t\eta&\,^tc&-\,^td\\
-\,^t\beta&\,^t\eta&\,^td&-\,^tc\\
  \end{pmatrix}.
\end{equation*} Now we make the following
\begin{claim}
  If $\alpha$ or $\gamma$ have a null row, then $M$ is not simple.
\end{claim}
In fact, assume $(\alpha_{11}, \ldots, \alpha_{1n})=0$. We have
$a_{12}\cdot x_1=\sum_{j}\beta_{1j}(v_j-w_j)$, if this is zero, then
$\langle x_1\rangle\cong S_\eps\subset M$ and $M$ is not simple. If
not, let $$\bar{v}_1=\sum_{j}\beta_{1j} v_j,\quad 
\bar{w}_1=\sum_{j}\beta_{1j} w_j.$$ Thus, $a_{12}\cdot
x_1=\bar{v}_1-\bar{w}_1$ and as $0=a_{12}^2x_1$ we have that
$a_{12}\bar{v}_1=a_{12}\bar{w}_1$. But, moreover, we also have that
$$a_{12}\bar{v}_1=\sum_i(\beta a)_{1i}
x_i+\sum_k(\beta(c+d))_{1k}(v_k+w_k)=0,$$ since $\beta a=0$ and
$(\beta(c+d))_{1k}=-(\alpha\eta)_{1k}=-\sum_l\alpha_{1l}\eta_{lk}=0$.
Then $\bar{v}_1=0$, $S_\eps\subset M$ and
$M$ is not simple. 

The claim when a row of $\gamma$ is null follows
analogously, or just tensoring with the representation $S_{\sg}$,
since it interchanges the roles of $\alpha$ and $\gamma$.

Then we see that, for $M$ to be simple, we necessarily must have
$\,^t\alpha$, $\,^t\gamma$ injective. But
$0=\,^t(\alpha\gamma)=\,^t\gamma\,^t\alpha\Rightarrow \alpha=0$.
Thus $M$ cannot be simple if $n,m>0$. Therefore, we are left with
the (equivalent) cases \begin{align*}
&M_{|\s_3}=M[\eps]\oplus
M[\st], &&\text{with } \dim M[\eps]=n, \quad \dim M[\st]=p, \quad n,p>0;\\
&M_{|\s_3}= M[\sg]\oplus
M[\st], &&\text{with } \dim M[\sg]=m, \quad \dim M[\st]=p, \, m,p>0.
                       \end{align*}
Assume we are in the
first case.
Thus, the equations above become:
\begin{equation}\label{eqn:sistema}
  \begin{cases}
    & a\beta=\beta a=0,\quad \beta(c+d)=0,\quad (c-d)a=0,\\
& d^2=c^2,\quad cd=dc, \quad c(-5c-4d)=\id.
  \end{cases}
\end{equation}
Now, in particular, if $\,^t\beta$ is injective, we have $\,^ta=0$
and thus $\mA_1\cdot M[\st]\varsubsetneq M[\st]$. But if $\,^t\beta$
is not injective, we may find a non-trivial linear combination $x$
of the elements $\{x_i\}_{i=1}^n$ making $S_\eps=\langle
x\rangle$ into an $\mA_1$-submodule of $M$.
\end{proof}

\subsubsection{Some indecomposable $\mA_1$-modules}

\ 

We start by studying the 3-dimen\-sional
indecomposable modules. As said in Lemma \ref{lem:indesc0}, it follows that for
such a module $M$, it holds either that $M_{|\s_3}\cong W_{\eps}\oplus W_{\st}$ or
$M_{|\s_3}\cong W_{\sg}\oplus W_{\st}$. Take $x,y,v,w$ such that
$\langle x\rangle _{|\s_3}=W_{\eps}$, $\langle y\rangle _{|\s_3}=W_{\sg}$,
$\langle v,w\rangle _{|\s_3}=W_{\st}$.

\begin{lem}\label{lem:indesc}
There are exactly eight non-isomorphic non-simple indecomposable
$\mA_1$-modules of dimension 3:
\begin{align}
\tag{i} &M_{\st,\eps}[\pm\frac\im3]=\k\{x,v,w\},&& a_{12}\cdot
v=\pm\frac\im3(v+w)+x,&& a_{12}\cdot x=0;\\
\tag{ii} &M_{\st,\sg}[\pm\im]=\k\{y,v,w\},&& a_{12}\cdot v=\pm\im(v-w)+y,&&
a_{12}\cdot y=0;\\
\tag{iii} &M_{\eps,\st}[\pm\im]=\k\{x,v,w\},&& a_{12}\cdot
v=\pm\im(v-w),&& a_{12}\cdot x=v-w;\\
\tag{iv} &M_{\sg,\st}[\pm\frac\im3]=\k\{y,v,w\},&& a_{12}\cdot
v=\pm\frac\im3(v+w),&& a_{12}\cdot y=v+w.
\end{align}
\end{lem}
\begin{proof}
It is straightforward to check that the listed objects are in fact $\mA_1$-modules and that they are not isomorphic to each other. 
Now, assume $M_{|\s_3}=W_{\eps}\oplus W_{\st}$, the other case being analogous. If $M$ is not simple,
then there is $N\subset M$ and necessarily $N_{|\s_3}=W_{\st}$ or
$N_{|\s_3}=W_{\eps}$. Then, the lemma follows specializing the
equations in \eqref{eqn:sistema1} to this case.
\end{proof}
\begin{pro}\label{pro:indesc2+1}
Let $M$ be an indecomposable non-simple $\mA_1$-module such that
$M_{|\s_3}=M[\eps]\oplus
M[\st]$, with $\dim M[\eps]=p$, $\dim M[\st]=q$ or
$M_{|\s_3}=M[\sg]\oplus
M[\st]$, with $\dim M[\sg]=p$, $\dim M[\st]=q$ for $p,q>0$. Then
$p=q=1$ and $M$ is isomorphic to one and only one of the modules defined
in Lemma \ref{lem:indesc}.
\end{pro}
\begin{proof}
We work with the case $M_{|\s_3}=M[\eps]\oplus
M[\st]$, with $\dim M[\eps]=p$, $\dim M[\st]=q$,
$p,q\geq 1$, the other resulting from this one by tensoring with
$S_{\sg}$. Let $M[\eps]=\k\{x_i\}_{i=1}^p$, $M[st]=\k\{v_i,w_i\}_{i=1}^q$
and $a, \beta, c, d$ be as in the proof of Th.
\ref{teo:simple s3}. Recall that they satisfy the system of
equations \eqref{eqn:sistema}. The last three
conditions from that system imply, as in the proof of Prop.
\ref{pro:suma de st}, that $c, d$ may be chosen as
\begin{equation*}
 \begin{array}{cc}
  c=\begin{pmatrix}
    \delta & 0\\ 0 &\delta'
  \end{pmatrix}, &
d=\begin{pmatrix}
   -\delta & 0\\ 0 &\delta'
  \end{pmatrix},
 \end{array}
\end{equation*}
for $\delta\in\k^{q_1\times q_1}$, $\delta'\in\k^{q_2\times q_2}$ diagonal
matrices with eigenvalues in $\{\pm\im\}$ and $\{\pm\dfrac\im3\}$, respectively,
$q_1+q_2=q$. Consequently,
\begin{equation*}
 \begin{array}{cc}
  \beta=\begin{pmatrix}
    \beta_1 & 0\\ \beta_2 &0
  \end{pmatrix}, &
a=\begin{pmatrix}
   0 & 0\\ a_1 & a_2
  \end{pmatrix}, \text{ with } a_1\beta_1+a_2\beta_2=0,
 \end{array}
\end{equation*}
\begin{equation*}
   a_{12}=\begin{pmatrix}
    0 & 0 & 0 &\,^ta_1 & 0& -\,^ta_1 \\
    0 & 0 & 0 &\,^ta_2 & 0& -\,^ta_2 \\
    \,^t\beta_1 & \,^t\beta_2 & \delta & 0 & \delta &0 \\
   0 & 0 & 0 & \delta' & 0& -\delta' \\
    -\,^t\beta_1 & -\,^t\beta_2 & -\delta & 0 & -\delta &0 \\
   0 & 0 & 0 & \delta' & 0& -\delta' \\
  \end{pmatrix}.
\end{equation*}
Assume $q_2>0$. In this case,
$\tilde{a}=\left(\begin{smallmatrix}\,^ta_1 \\ \,^ta_2
\end{smallmatrix}\right)$ must be injective. Otherwise, we may
change the basic elements $\{v_{q_1+1},\ldots, v_{q},
w_{q_1+1},\ldots, w_{q}\}$ in such a way that, for some $q_1+1\leq
r< q$, the last $q-r$ columns of $\tilde{a}$ are null and in that
case $$M=\langle v_{q_1-r+1},\ldots,v_q\rangle\oplus\langle x_i,v_j:
i=1,\ldots,p; j=1,\ldots, q-r\rangle.$$ Thus $\tilde{a}$ is injective.
Change the basis $\{x_i:i=1,\ldots, p\}$ in such a way that $$a_{12}\cdot
v_{q_1+i}=x_i+\frac\im3(v_{q_1+i}+w_{q_1+i}), \quad i=1,\ldots, q_2.$$
Notice that, as $a_{12}(v_{q_1+i}+w_{q_1+i})=0$ for every $i$ and
$a_{12}^2=0$, then $a_{12}\cdot x_i=0$, $i=1,\ldots, q_2$. But then $$
M=\bigoplus_{i=1}^{q_2}\langle x_i, v_{q_1+i}\rangle\oplus\langle
x_{q_2+1},\ldots, x_p, v_1,\ldots, v_{q_1}\rangle.$$
Therefore, if $q_2>0$ and $M$ is indecomposable, then
$q_1=0$, $p=q_2=1$, and this gives us the modules in the first item of Lemma
\ref{lem:indesc}.

Analogously, if $q_1>0$,
$\tilde{\beta}=\left(\begin{smallmatrix}\,^t\beta_1 & \,^t\beta_2
\end{smallmatrix}\right)$ must be injective, and $q_2=0$. If
$v_1,\ldots, v_p$ are chosen in such a way that $a_{12}\cdot
x_i=v_i-w_i$, $i=1,\ldots, p$, then $ M=\bigoplus_{i=1}^p\langle
x_i,v_i\rangle\oplus\bigoplus_{i=p+1}^{q_1}\langle v_i\rangle
$
and therefore $p=q_1=1$, giving the modules in the third item of the lemma. The
modules in the other two items result from these ones by tensoring with
$S_{\sg}$.
\end{proof}

\subsubsection{Tensor product of simple $\mA_1$-modules}
Here we compute the tensor product of two given simple
$\mA_1$-modules, and show that it turns out to be an indecomposable
module.

First, we list all of the indecomposable $\mA_1$-modules of
dimension 4. Notice that if $M$ is such an indecomposable module,
then we necessarily must have $M_{|\s_3}=W_{\eps}\oplus
W_{\sg}\oplus W_{\st}$, by Props. \ref{pro:suma de st} and
\ref{pro:indesc2+1}. In the canonical basis, the matrix of $a_{12}$ is
given by
$$[a_{12}]=\begin{pmatrix}0 & \gamma & a&-a\\
\alpha& 0& b&b\\
\beta&\eta&c&-d\\
-\beta&\eta&d&-c\\
  \end{pmatrix},$$
for some $\alpha,\gamma,a,b\in\k$ and $c=d=\pm\frac\im3$ or
$c=-d=\im$. For every $c=\theta\in\{\pm\im,\pm\frac\im3\}$ and for
each collection $(\alpha,\beta,\gamma,\eta,a,b)$ which defines
representation, we denote by
$M(\alpha,\beta,\gamma,\eta,a,b)[\theta]$ the corresponding module.

\medbreak

\begin{pro}\label{pro:indesc4}\quad
\noindent
\begin{itemize}
\item Let $\theta=\pm\frac\im3$. There are exactly four non-isomorphic
indecomposable modules
$M(\alpha,\beta,\gamma,\eta,a,b)[\pm\frac\im3]$. They are defined
for $(\alpha,\beta,\gamma,\eta,a,b)$ in the following list:
\begin{enumerate}
\smallskip
  \item $(0,0,1,0,1,0)$,
\smallskip
  \item $(0,0,1,1,0,0)$,
\smallskip
  \item $(1,0,0,0,\mp\frac{2\im}3,1)$,
\smallskip
  \item $(1,1,0,\mp\frac{2\im}3,0,0)$.
\end{enumerate}
\smallskip
  \item Let $\theta=\pm\im$. There are exactly four non-isomorphic
indecomposable modules $M(\alpha,\beta,\gamma,\eta,a,b)[\pm\im]$.
They are defined for  $(\alpha,\beta,\gamma,\eta,a,b)$ in the
following list:
\begin{enumerate}
\smallskip
\item $(1,0,0,0,0,1)$,
  \smallskip
\item $(1,1,0,0,0,0)$,
  \item $(0,\mp2\im,1,1,0,0)$,
\smallskip
  \item $(0,0,1,0,1,\mp2\im)$.
\end{enumerate}
\end{itemize}
\end{pro}
The next proof is essentially interpreting the equations \eqref{eqn:sistema1}
in this case.
\begin{proof}
We have the following identities
\begin{equation}\label{eqn:identities}
  \alpha\gamma=\gamma\alpha=0,\quad
\beta a=\beta b=\eta a=\eta b=0.
\end{equation}
Assume $c=d=\pm\frac\im3$, then to the equations listed above we must add:
\begin{equation*}
 0=2\beta c+\alpha\eta=a\alpha+2cb, \quad
0=\gamma\beta=b\gamma.
\end{equation*}
We compute the solutions. Notice that  $\alpha=0\Rightarrow\beta=0\Rightarrow b=0\Rightarrow\eta a=0$.
Then according to $\eta=0$ or $a=0$ we have:
\begin{equation*}
 \begin{array}{cc}
  \begin{cases}
   a_{12}\cdot x=0, \\
a_{12}\cdot y=\gamma x,\\
 a_{12}\cdot v=ax+c(v+w)
  \end{cases}
& \text{or}\quad
  \begin{cases}
   a_{12}\cdot x=0,\\ a_{12}\cdot y=\gamma x+\eta(v+w),\\
a_{12}\cdot v=c(v+w).
  \end{cases}
 \end{array}
\end{equation*}
Notice that, in any case, we cannot have $\gamma=0$, otherwise the
module would decompose. We may thus assume $\gamma=1$, changing $y$
by $\frac1\gamma y$. For the same reason, we cannot have $a=\eta=0$.
In the first case, we may take $a=1$, changing $v$ by $\frac1av$ and
in the second case, changing $v$ by $\eta v$ we may take $\eta=1$.

On the other hand, $\gamma=0\Rightarrow \alpha\neq 0$; and,
according to $\beta=0$ or $\beta\neq 0$,
\begin{align*}
\beta&=0\Rightarrow\begin{cases}
  a_{12}\cdot x=\alpha y, \\
   a_{12}\cdot y=0\\
   a_{12}\cdot v=ax+by+c(v+w),\quad  \text{for } a=-2cb\alpha^{-1}
\end{cases}\\
\\
\beta &\neq 0\Rightarrow \begin{cases}
 a_{12}\cdot x=\alpha y+\beta (v-w), \\
 a_{12}\cdot y=\eta(v+w),\\
  a_{12}\cdot v=c(v+w),\hskip2.1cm  \text{for }
\eta=-2\beta c\alpha^{-1}.
\end{cases}
\end{align*}
In the first case we may assume $\alpha=b=1$, and thus $a=-2c$ and,
in the second, $\alpha=\beta=1$, and thus $\eta=-2c$.

Assume now $c=-d=\pm\im$, then to the identities \eqref{eqn:identities} we had
we must add:
\begin{equation*}
  \begin{cases}
0=2b\gamma+2ca=\gamma\beta+2c\eta\\
0=a\alpha=\alpha\eta.
 \end{cases}
\end{equation*}

We find the solutions:
 \begin{align*}
&  \text{(i)} \begin{cases}
  a_{12}\cdot x=\alpha y,\\  a_{12}\cdot y=0,\\
a_{12}\cdot
v=by+c(v-w).
       \end{cases}
     && \text{(ii)} \begin{cases}  a_{12}\cdot x=\alpha y+\beta(v-w),\\ 
a_{12}\cdot
y=0,\\
a_{12}\cdot
v=c(v-w).
\end{cases}  \\ & && \\
   &  \text{(iii)} \begin{cases} a_{12}\cdot x=\beta(v-w),\\  a_{12}\cdot
y=\gamma
x+\eta(v+w),\\
     a_{12}\cdot v=c(v-w), \\
\beta=-2\eta
    c\gamma^{-1}.
\end{cases}
 & &     \text{(iv)} \begin{cases} a_{12}\cdot x=0,\\  a_{12}\cdot y=\gamma x,\\
  a_{12}\cdot v=ax+by+c(v-w),
\\
    b=-2ca\gamma^{-1}.
\end{cases}
            \end{align*}
Therefore, changing conveniently the basis on each case (by scalar
multiple of its components), we have the four modules from the
second item.
\end{proof}

Let $\sg:\im\mathbb{R}\to\{\pm 1\}$, $\sg(\im t)=\sg(t)$.

\begin{pro}\label{pro:tensor}
The following isomorphisms hold:
\begin{enumerate}
  \item $S_\eps\ot S\cong S\cong S\ot S_\eps$ for every simple
$\mA_1$-module $S$;
\item $S_{\sg}\ot S_{\st}(\theta)\cong S_{\st}(\vartheta)$, for
$\theta,\vartheta\in\{\pm\im,\pm\frac\im3\}$ with
$\sg(\theta)=\sg(\vartheta)$, $|\theta|\neq|\vartheta|$;
\item $S_{\st}(\theta)\ot S_{\sg}\cong S_{\st}(\vartheta)$, for
$\theta,\vartheta\in\{\pm\im,\pm\frac\im3\}$ with
$\sg(\theta)=-\sg(\vartheta)$, $|\theta|\neq|\vartheta|$.
\item \begin{itemize}
 \item $ S_{\st}(\im)\ot S_{\st}(\im)\cong S_{\st}(-\frac\im3)\ot
S_{\st}(\frac\im3)\cong M(0,2\im,1,1,0,0)[-\im]$,
\vskip.2cm
 \item $ S_{\st}(\im)\ot S_{\st}(-\im)\cong S_{\st}(-\frac\im3)\ot
S_{\st}(-\frac\im3)\cong M(1,0,0,0,-2\frac\im3,1)[\frac\im3]$,
\vskip.2cm
 \item $ S_{\st}(\im)\ot S_{\st}(\frac\im3)\cong S_{\st}(-\frac\im3)\ot
S_{\st}(\im)\cong M(0,0,1,0,1,2\im)[-\im]$,
\vskip.2cm
 \item $ S_{\st}(\im)\ot S_{\st}(-\frac\im3)\cong S_{\st}(-\frac\im3)\ot
S_{\st}(-\im)\cong M(1,1,0,-2\frac\im3,0,0)[\frac\im3]$,
\vskip.2cm
 \item $ S_{\st}(-\im)\ot S_{\st}(\im)\cong S_{\st}(\frac\im3)\ot
S_{\st}(\frac\im3)\cong M(1,0,0,0,2\frac\im3,1)[-\frac\im3]$,
\vskip.2cm
\item $ S_{\st}(-\im)\ot S_{\st}(-\im)\cong S_{\st}(\frac\im3)\ot
S_{\st}(-\frac\im3)\cong M(0,-2\im,1,1,0,0)[\im]$,
\vskip.2cm
\item $ S_{\st}(-\im)\ot S_{\st}(\frac\im3)\cong S_{\st}(\frac\im3)\ot
S_{\st}(\im)\cong M(1,1,0,2\frac\im3,0,0)[-\frac\im3]$,
\vskip.2cm
\item $
S_{\st}(-\im)\ot S_{\st}(-\frac\im3)\cong S_{\st}(\frac\im3)\ot
S_{\st}(-\im)\cong M(0,0,1,0,1,-2\im)[\im]$.
\end{itemize}
\end{enumerate}
\end{pro}
\begin{proof}
Item (i) is immediate. 

We check item (ii): let
$\theta\in\{\pm\im,\pm\frac\im3\}$,
$S_{\sg}=\k\{z\}$; $S_{\st}(\theta)=\k\{v,w\}$, $a_{12}\cdot v=cv+dw$. Then
$(S_{\sg}\ot S_{\st})_{|\s_3}=W_{\st}$
with the canonical basis given by $$u=z\ot v-2z\ot w, \quad t=2z\ot
v-z\ot w, $$
and then
$$
a_{12}u=\frac{5c+4d}3u-\frac{4c+5d}3t.
$$
Thus, the claim
follows according to $c=\pm\im$ or $c=\pm\frac\im3$. 

Item (iii) follows
analogously: in this case $$u=v\ot
z-2w\ot z\quad \text{and}\quad a_{12}u=-\frac{5c+4d}3u+\frac{4c+5d}3t.$$

Now, we have to compute $S_{\st}(\theta)\ot S_{\st}(\vartheta)$, for
$\theta,\vartheta\in\{\pm\im,\pm\frac\im3\}$. Let
$S_{\st}(\theta)=\k\{v,w\}, S_{\st}(\vartheta)=\k\{v',w'\}$, $a=v\ot
v', b=v\ot w', c=w\ot v', d=w\ot w'$. First,
\begin{align*}
W_{\st}\ot
W_{\st}&\cong W_{\eps}\oplus W_{\sg}\oplus W_{\st}
= \k\{x\}\oplus \k\{y\}\oplus \k\{v,w\},
\end{align*}
for $x=2a-b-c+2d$, $y=b-c$, $v=a-b-c, w=d-b-c$. Now, if $a_{12}\cdot
v=\alpha v+\beta w$ and $a_{12}\cdot v'=\alpha' v'+\beta' w'$, then
\begin{align*}
  &a_{12}\cdot a=\alpha a+(\beta+\alpha')c+\beta'd,\quad a_{12}\cdot b=\alpha
b-\beta'c+(\beta-\alpha')d,\\
&a_{12}\cdot c=(\alpha'-\beta)a+\beta'b-\alpha c,\quad a_{12}\cdot
d=-\beta'a-(\alpha'+\beta)b-\alpha d;\\
\end{align*}
 and thus
\begin{align*}
  &a_{12}\cdot
x=(-\alpha-2\beta-2\alpha'-\beta')y+(2\alpha+\beta-\alpha'-2\beta')(v-w),\\
&a_{12}\cdot
y=\frac13(\alpha+2\beta-2\alpha'-\beta')x+(-2\alpha-\beta+\alpha'+2\beta')(v+w),
\\
&a_{12}\cdot
v=\frac16(2\alpha+\beta+\alpha'+2\beta')x+\frac12(-2\alpha-\beta-\alpha'-2\beta'
)y\\
&\qquad\qquad\qquad+\frac13(\alpha+2\beta-4\alpha'-2\beta')v+\frac13(-2\alpha-4
\beta+2\alpha'+\beta')w.
\end{align*}
For each $\theta,\vartheta\in\{\pm\im\pm\frac\im3\}$, we get the
identities in item (iv) by inserting the corresponding values of
$\alpha,\alpha',\beta,\beta'$.
\end{proof}
\begin{cor}
 $\mA_1$ is not quasitriangular.
\end{cor}
\begin{proof}
 If $H$ is a quasitriangular Hopf algebra and $M,N$ are $H$-modules, then $M\ot
N\cong N\ot M$ as $H$-modules. We see that this does not hold for $\mA_1$, from,
for instance, the second item of Prop. \ref{pro:tensor}.
\end{proof}

\subsubsection{Projective covers}
Recall that a
linear basis for $\mA_1$ is given by the set
$S=\{xH_t\,|\,x\in
 X,t\in\s_3\}$, where
$X=\{1,a_{12},a_{13},a_{23},a_{12}a_{13},a_{12}a_{23},a_{13}a_{23},\\
a_{13}a_{12},a_{12}a_{13}a_{23},a_{12}a_{13}a_{12},
a_{13}a_{12}a_{23},a_{12}a_{ 13}a_{12}a_{23}\}$ \cite{AG2}.
\begin{pro}\label{pro:projcover}
$  I_\chi$ is the projective cover of $S_\chi$,
$\chi\in\{\eps,\sg\}$.
\end{pro}
\begin{proof}
 In view of Prop \ref{pro:proj}, we only have to check that $I_\chi$ is
 indecomposable. We work with $\chi=\eps$, the other case being analogous, or follows by tensoring with $S_{\sg}$. Let $e_\eps=\sum_{t\in\s_3}H_t\in\mA_1$, then it
is clear that $ \{x e_\eps\,|\,x\in X\}$ is a basis of
$I_\eps$. Moreover, if we change this basis by the following one:
\begin{align*}
\{e_\eps\}&\cup\{(a_{12}a_{13}a_{12}a_{23}-a_{12}a_{23})e_{\eps}\}\cup\{(a_{12}
+a_{13}+a_{23})e_\eps\}\\
&\cup\{(a_{12}a_{13}a_{12}-a_{12}a_{13}a_{23}-a_{13}a_{12}a_{23}-a_{13}-2a_{12}
)e_{\eps}\}\\
&\cup\{(a_{12}-2a_{13}+a_{23})e_\eps,(2a_{23}-a_{12}-a_{13})e_\eps\}\\
&\cup\{(a_{13}a_{23}-a_{13}a_{12})e_\eps,(a_{12}a_{13}-a_{12}a_{23}+a_{13}a_{23}
-a_{13}a_{12})e_\eps\}\\
&\cup\{(a_{12}a_{13}+a_{12}a_{23}+a_{13}a_{12})e_\eps,(-a_{12}a_{13}+a_{13}a_{23
}-a_{13}a_{12})e_{\eps}\}\\
&\cup\{(a_{12}a_{13}a_{12}+2a_{12}a_{13}a_{23}-a_{13}a_{12}a_{23}+a_{12}-a_{13}
)e_{\eps},\\
&\qquad(2a_{12}a_{13}a_{12}+a_{12}a_{13}a_{23}+a_{13}a_{12}a_{23}-a_{12}+a_{13}
)e_{\eps}\}
\end{align*}
then we can see that $$(I_\eps)_{|\s_3}\cong W_\eps\oplus W_\eps\oplus
W_{\sg}\oplus W_{\sg}\oplus W_{\st}\oplus W_{\st}\oplus W_{\st}\oplus
W_{\st}.$$ Now we deal with the action of
$a_{12}$. Notice that in the first basis, the matrix of $a_{12}$ is
$E_{2,1}+E_{5,3}+E_{6,4}+E_{10,7}+E_{9,8}+E_{12,11}$, where
$E_{i,j}$ is the matrix whose all its entries are zero except for
the $(i,j)$-th one, which is a 1. It is possible to change the
basis in such a way that the decomposition in $\s_3$-simple modules
is preserved and the matrix of $a_{12}$ becomes:

\begin{equation*}
 [a_{12}]=\left[\begin{array}{cccccccccccc}
0& 0& 0& 0& 0& 0& 0& 0& 0& 0& 0& 0\\
0& 0& 0& -1& -1& 1& -1& 1& -1& 1& -1& 1\\
\frac13& 0& 0& 0& 0& 0& 0& 0& 0& 0& 0& 0\\
0& 0& 0& 0& -2 \im&-2 \im& 2 \im& 2 \im& 0& 0& 0& 0\\
-\frac1{12}& 0& 0& 0& \im& \im& 0& 0& 0& 0& 0& 0\\
\frac1{12}& 0& 0& 0& -\im& -\im& 0&0& 0& 0& 0& 0\\
-\frac1{12}& 0& 0& 0& 0& 0& -\im& -\im& 0& 0& 0& 0\\
\frac1{12}& 0& 0& 0& 0& 0& \im&\im& 0& 0& 0& 0\\
\frac1{12}& 0& -\frac\im6& 0& 0& 0& 0& 0& \frac\im3& -\frac\im3& 0& 0\\
-\frac1{12}& 0& -\frac\im6& 0& 0& 0& 0& 0& \frac\im3& -\frac\im3& 0& 0\\
\frac1{12}& 0& \frac\im6& 0& 0& 0& 0& 0& 0& 0& -\frac\im3& \frac\im3\\
-\frac1{12}& 0& \frac\im6& 0& 0& 0& 0& 0& 0& 0& -\frac\im3&
\frac\im3
\end{array}\right].
\end{equation*}

Let $\{x_1,x_2,y_1,y_2,v_1,w_1,v_2,w_2,v_3,w_3,v_4,w_4\}$ be this
new basis. Assume $I_\eps=U_1\oplus U_2$, for $U_1,U_2$
$\mA_1$-submodules. Thus, there exists $i=1,2$,
$\lambda\neq 0, \mu\in\k$ such that $x=\lambda x_1+\mu x_2\in U_i$.
Acting with $a_{12}$ we have that
$y_1,v_1+v_2-v_3-v_4\in U_i$. As $y_1\in U_i$, acting once again
with $a_{12}$ we have that also $v_3-v_4\in U_i$ and thus
$v_3+v_4\in U_i$ (again by the action of $a_{12}$). Therefore
$v_3,v_4\in U_i$ and so $x_2,y_2,x_1,v_1+v_2\in U_i$. But then
$v_1-v_2\in U_i$ and thus $U_i=I_\eps$.
\end{proof}

We are left with finding the projective covers $P_{\st}(\theta)$ of
the 2-dimensional $\mA_1$-modules $S_{\st}(\pm\theta)$,
$\theta\in\{\im,\frac\im3\}$. Since these modules
are
\begin{equation*}
  S_{\st}(\im),\quad  S_{\st}(\im)\ot S_{\sg},\quad S_{\sg}\ot
S_{\st}(\im),\quad\text{and}\quad S_{\sg}\ot S_{\st}(\im)\ot
S_{\sg},
\end{equation*}
see Prop. \ref{pro:tensor}, and $P_{\st}(\theta)\cong \mA_1
e_{\st}(\theta)$, they will all have the same dimension. Moreover,
we will necessarily have $\dim P_{st}(\theta)=6$, $\forall\,\theta$,
by \eqref{eqn:proj}.

\begin{pro}
 Let $P$ be the $\k\s_3$-module with basis $\{x,y,u,t,v,w\}$, where 
$
\langle x\rangle _{|\s_3}=W_\eps,\quad\langle y\rangle
_{|\s_3}=W_{\sg},\quad\langle u,t \rangle _{|\s_3}=W_{\st},\quad
\langle v,w\rangle _{|\s_3}=W_{\st}.
$
Then $P$ is an $\mA_1$-module via
$$
\k\{x,y,u,t\}\cong M(0,2\im,1,1,0,0)[-\im],\quad a_{12}\cdot
v=x-2\im y+u+t+\im(v-w).
$$
Moreover $P=P_{\st}(\im)$ is the projective cover of the simple
module $S_{\st}(\im)$.

\medbreak

As a result, we have
$P_{\st}(-\frac\im3)=P_{\st}(\im)\ot S_{\sg}$,
$P_{\st}(\frac\im3)=S_{\sg}\ot P_{\st}(\im)$ and
$P_{\st}(-\im)=S_{\sg}\ot P_{\st}(\im)\ot S_{\sg}$.
\end{pro}
\begin{proof}
The matrix of
$a_{12}$ in the given basis is
\begin{equation*}
 [a_{12}]=\left(\begin{array}{cccccc}
   0& 1& 0& 0& 1& -1\\
 0& 0& 0& 0& -2 \im& -2 \im\\
 2 \im& 1& -\im& -\im& 1& -1\\
 -2 \im& 1& \im& \im& 1& -1\\
 0& 0& 0& 0& \im& \im\\
 0& 0& 0& 0& -\im& -\im
 \end{array}\right).
\end{equation*}
Via the action of $H_{13}, H_{23}$ we define the matrices of
$a_{13},a_{23}$ and then it is easy to check that
\begin{align*}
  [H_{12}][a_{12}]&=-[a_{12}][H_{12}],\\
  [a_{12}]^2&=0 \\
 [a_{12}][a_{13}]+[a_{13}][a_{23}]+[a_{23}][a_{12}]&=\id_{6\times
6}-[H_{12}][H_{12}],
\end{align*}
and thus $P$ is an $\mA_1$-module.

Now, it is clear that $U=\k\{x,y,u,t\}$ is an $\mA_1$-submodule and
that the canonical projection $\pi:P\twoheadrightarrow P/U$ gives a
surjection over $S_{\st}(\im)$. Moreover, this surjection is
essential. In fact, let $N\subset P$ be an $\mA_1$-submodule, such
that $N/U\cong S_{\st}(\im)$. In particular, there exists
$\lambda\neq 0\in\k$ such that $\lambda u+ v\in P$. Now,
$
a_{12}(v+\lambda u)=x-2\im
y+(1-\lambda\im)u+(-1+\lambda\im)t+\im(v-w),
$
and thus $x,y\in N$. But $x\in N\Rightarrow u,v\in N$ and therefore
$N=P$. Consequently, $\pi:P\to P/U$ is essential.

Now, if $(P_{\st}(\im),f)$ is the projective cover of
$S_{\st}(\im)$, we have the following commutative diagram
\begin{equation*}
 \xymatrix{
  & & P_{\st}(\im)\ar[d]^f\ar@{-->}[lld]_{g}\\
  P\ar@{->>}[r]^{\pi} & P/U\ar@{-}[r]^{\cong} & S_{\st}(\im).}
\end{equation*}
As $\pi$ is essential and $\pi(g(P_{\st}(\im)))\cong
S_{\st}(\im)$ we must have $g(P_{\st}(\im))=P$. But then $\dim
P=\dim P_{\st}(\im)=6$ and thus $g$ is an isomorphism.
Therefore, $(P,\pi)$ is the projective cover of $S_{\st}(\im)$. The
claim about the projective covers of the other $S_{\st}(\lambda)$'s is now
straightforward.
\end{proof}

\subsubsection{Representation type of $\mA_1$}\label{subsub:reprtypea1}

\

We show that the algebra $\mA_1$ is not of
finite representation type. From Props. \ref{pro:1sn} and \ref{pro:suma de st} it follows that
$\ext_{\mA_1}^1(S,S)=0$ for any simple one-dimensional
$\mA_1$-module $S$, and that there is an unique non-trivial
extension of $S_\eps$ by $S_{\sg}$, namely the $\mA_1$-module
$M_{\sg,\eps}$. The same holds for extensions of $S_{\sg}$ by $S_\eps$,
considering the $\mA_1$-module $M_{\eps,\sg}$. Prop. 3.7 shows that
$\ext_{\mA_1}^1(S_{\st}(\lambda),S_{\st}(\mu))=0$ for any
$\lambda,\mu\in\{\pm\im,\pm\frac\im3\}$. Now, a non-trivial
extension of one of the modules $S_\eps$ or $S_{\sg}$ by a two
dimensional $\mA_1$-module $S_{\st}(\lambda)$, or vice versa, must
come from a three dimensional indecomposable $\mA_1$-module $M$. We
have classified such modules in Lemma \ref{lem:indesc} and we see then that:

\begin{align*}
\dim\ext_{\mA_1}^1(S_\eps,S_{\st}(\lambda))=\dim\ext_{\mA_1}^1(S_{\st}(\lambda),
S_{\sg})&=\begin{cases}
    1, & \mbox{if} \ \lambda=\pm\im,\\
    0, & \mbox{if} \ \lambda=\pm\frac\im3.\\
      \end{cases}\\ 
\dim\ext_{\mA_1}^1(S_{\sg},S_{\st}(\lambda))=\dim\ext_{\mA_1}^1(S_{\st}(\lambda)
,S_\eps)&=\begin{cases}
1, & \mbox{if} \ \lambda=\pm\frac\im3,\\
0, & \mbox{if} \ \lambda=\pm\im.\\
      \end{cases}
\end{align*}
Let
$\{S_\eps,S_{\sg},S_{\st}(\im),S_{\st}(-\im),S_{\st}(\frac\im3),S_{\st}
(-\frac\im3)\}=\{1,2,3,4,5,6\}$
be an ordering of the simple $\mA_1$-modules. Then the Ext-Quiver of
$\mA_1$ is:

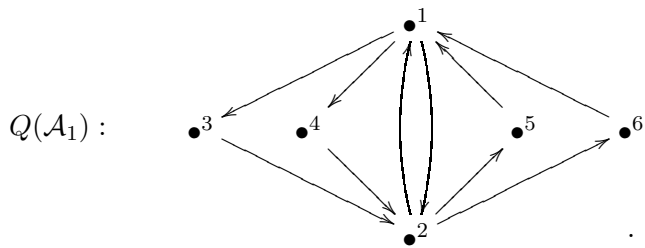
\begin{equation*}
 \xymatrix{& & & \bullet^1\ar[lld]\ar[ld]\ar@/^/[dd] & & \\
 Q(\mA_1):&\bullet^3\ar[rrd] & \bullet^4\ar[rd] & &\bullet^5\ar[lu] &
\bullet^6\ar[llu]\\
 && & \bullet^2\ar[ru]\ar[rru]\ar@/^/[uu] & & .}
\end{equation*}

\begin{pro}
$\mA_1$ is not of finite representation type.
\end{pro}

\begin{proof}
The separation diagram of $\mA_1$ is $D_5^{(1)}\coprod D_5^{(1)}$,
with $D_5^{(1)}$ the extended affine Dynkin diagram corresponding to
the classical Dynkin diagram $D_5$.
By Lemma \ref{lem:extq} we have that $\mA_1/J(\mA_1)^2$ (a
quotient of $\mA_1$) is not of finite representation type (it is, in
fact, tame) by Th. \ref{teo:aus}, and so neither is $\mA_1$.
\end{proof}

\subsection*{Acknowledgements} I thank my advisor Nicol\'as Andruskiewitsch for his
many suggestions and the careful reading of
this work. I also thank Gast\'on Garc\'ia for fruitful discussions at
early stages of the work. I thank Mar\'ia In\'es Platzeck for enlightening
conversations.


\begin{thebibliography}{XXXXX}
\bibitem[AG1]{AG1} {\sc Andruskiewitsch, N.} and {\sc Gra\~na, M.},
\emph{From racks to pointed Hopf algebras}, Adv. in Math.
\textbf{178} (2), 177--243 (2003).
\bibitem[AG2]{AG2} {\sc Andruskiewitsch, N.} and {\sc Gra\~na, M.},
\emph{Examples of liftings of Nichols algebras over racks},
Theories d'homologie, representations et algebres de Hopf, AMA
Algebra Montp. Announc.  2003, Paper 1, 6 pp. (electronic).
\bibitem[AHS]{AHS} {\sc Andruskiewitsch, N.}, {\sc Heckenberger, I.}
and {\sc Schneider, H.J.}, \emph{The Nichols algebra of a semisimple
Yetter-Drinfeld module}, arXiv:0803.2430v1.
\bibitem[ARS]{ARS} {\sc Auslander, M.}, {\sc
Reiten, I.} and {\sc Smal\o, S.},\emph{Representation theory of
Artin algebras}, Cambridge studies in advanced mathematics
\textbf{36}.
\bibitem[AS]{AS}  N. Andruskiewitsch and  H.-J. Schneider,
    \emph{Pointed Hopf Algebras}, in ``New directions in Hopf algebras'',
    1--68, Math. Sci. Res. Inst. Publ. \textbf{43}, Cambridge Univ. Press,
    Cambridge, 2002.
\bibitem[AZ]{AZ} {\sc Andruskiewitsch, N.} and
{\sc Zhang, F.}, \emph{On pointed Hopf algebras associated to some
conjugacy classes in $\s_n$}, Proc. Amer. Math. Soc. \textbf{135}
(2007), 2723--2731.
\bibitem[CR]{CR} {\sc Curtis, C. W.} and {\sc Reiner, I.}, \textsl{Methods of
representation theory, with applications to finite groups and orders I}, Wiley
Classics Library, (1981).
\bibitem[GG]{GG} {\sc Garc\'ia, G. A.} and {\sc Garc\'ia Iglesias, A.},
\textsl{Pointed Hopf algebras over $\s_4$}. Israel Journal of Math.
Accepted. Also available at arXiv:0904.2558v1 [math.QA].
\bibitem[MS]{MS} {\sc Milinski, A.} and {\sc Schneider, H.J.},
\emph{Pointed indecomposable Hopf algebras over Coxeter groups},
Contemp. Math. 267, 215--236 (2000).
\end{thebibliography}
\end{document}